\documentclass[11pt, notitlepage]{article}
\usepackage{amssymb,amsmath,comment}
\catcode`\@=11 \@addtoreset{equation}{section}
\def\thesection{\arabic{section}}

\def\theequation{\thesection.\arabic{equation}}
\catcode`\@=12
\usepackage{colortbl}%
\usepackage{a4wide}

\newcommand{\ds} {\displaystyle}
\newcommand{\e}{\epsilon}
\newcommand{\pa} {\partial}
\newcommand{\al} {\alpha}
\newcommand{\ba} {\beta}
\newcommand{\de} {\delta}
\newcommand{\ga} {\gamma}

\newcommand{\Om} {\Omega}
\newcommand{\rp} {\rightharpoonup}
\newcommand{\ra} {\rightarrow}

\newcommand{\De} {\Delta}
\newcommand{\la} {\lambda}
\newcommand{\La} {\Lambda}
\newcommand{\noi} {\noindent}

\newcommand{\mb} {\mathbb}
\newcommand{\mc} {\mathcal}
\newcommand{\lra} {\longrightarrow}

\usepackage[all]{xy}
\catcode`\@=11
\def\theequation{\@arabic{\c@section}.\@arabic{\c@equation}}
\catcode`\@=12

\def\QED{\hfill {$\square$}\goodbreak \medskip}

\newtheorem{Theorem}{Theorem}[section]
\newtheorem{Lemma}[Theorem]{Lemma}
\newtheorem{Proposition}[Theorem]{Proposition}

\newtheorem{Definition}[Theorem]{Definition}

\begin{document}
{\vspace{0.01in}
	\title{Unbalanced $(p,2)$-fractional problems with critical growth}
	
	\author{{\bf Deepak Kumar\footnote{email: {\tt deepak.kr0894@gmail.com}}  \;and  \bf K. Sreenadh\footnote{	e-mail: {\tt sreenadh@maths.iitd.ac.in}}} \\ Department of Mathematics,\\ Indian Institute of Technology Delhi,\\
		Hauz Khaz, New Delhi-110016, India. }

	\date{}
	
	\maketitle

\begin{abstract}

We study the existence, multiplicity and regularity results of
non-negative solutions of following doubly nonlocal problem:
$$ (P_\la) \left\{
\begin{array}{lr}\ds
 \quad   (-\Delta)^{s_1}u+\ba (-\Delta)^{s_2}_{p}u = \la a(x)|u|^{q-2}u+ \left(\int_{\Om}\frac{|u(y)|^r}{|x-y|^{\mu}}~dy\right)|u|^{r-2} u \quad \text{in}\;
\Om, \\
 \quad \quad\quad \quad u =0\quad \text{in} \quad \mb R^n\setminus \Om,
\end{array}
\right.
$$
where $\Om\subset\mb R^n$ is a bounded domain with $C^2$ boundary $\pa\Om$, $0<s_2 < s_1<1$,  $n> 2 s_1$, $1< q<p< 2$, $1<r \leq 2^{*}_{\mu}$ with $2^{*}_{\mu}=\frac{2n-\mu}{n-2s_1}$, $\la,\ba>0$ and $a\in L^{\frac{d}{d-q}}(\Om)$, for some $q<d<2^{*}_{s_1}:=\frac{2n}{n-2s_1}$, is a sign changing function.  We prove that each nonnegative weak solution of $(P_\la)$ is bounded. Furthermore, we obtain some existence and multiplicity results using Nehari manifold method.

\medskip

\noi \textbf{Key words:} Non-local operator, Fractional
$(p,2)$ Laplacian, Choquard equation, Nehari manifold.

\medskip

\noi \textit{2010 Mathematics Subject Classification:} 35J35, 35J60,
35R11.

\end{abstract}

\section {Introduction}
\noi In this article, we are concerned with the regularity, existence and multiplicity results for solutions to $(p,2)$-fractional Choquard equation
\begin{equation*}
  (P_\la) \left\{
\begin{array}{lr}\ds
 \quad  (-\Delta)^{s_1}u+\ba(-\Delta)^{s_2}_{p}u = \la a(x)|u|^{q-2}u+ \left(\int_{\Om}\frac{|u(y)|^r}{|x-y|^{\mu}}~dy\right)|u|^{r-2} u \; \text{in}\;
\Om \\
 \quad \quad  u =0\quad \text{in} \quad \mb R^n\setminus \Om,
\end{array}
\right.
 \end{equation*}
\noi where $\Om$ is a bounded domain in $\mb R^n$ with $C^2$ boundary $\pa\Om$, $0<s_2 < s_1<1$,  $n> 2 s_1$, $1< q<p< 2$, $1<r \leq 2^{*}_{\mu}$ with $2^{*}_{\mu}=\frac{2n-\mu}{n-2s_1}$, $\la,\ba>0$, $0<\mu<n$ and $a\in L^{\frac{d}{d-q}}(\Om)$, where $q<d<2^{*}_{s_1}$ and $2^{*}_{s_1}=\frac{2n}{n-2s_1}$, is a sign changing function such that $a^+(x):=\max\{a(x),0 \}\not\equiv 0$. $(-\Delta)^{s}_{m}$  is the Fractional $m$-Laplace operator, defined as follows
\begin{equation*}
{(-\Delta)^{s}_mu(x)}=- 2\lim_{\e\ra 0}\int_{\mb R^n\setminus B_\e(x)} \frac{|u(y)-u(x)|^{m-2}(u(x)-u(y))}{|x-y|^{n+ms}}dy,
\end{equation*}
 for $m>1$ and $s\in(0,1)$. The problems involving these kind of operators have their applications in obstacle problems, conservation laws, phase transition, image processing, anomalous diffusion, American options in finance. For more details, we refer to \cite{ caffarelli,nezzaH,mingqi,pham}  and the references therein.\par
\noi In the local case, that is, when $s_1=s_2=1$, and $\ba=0$,
the following equation with Choquard type nonlinearity
 \begin{equation}\label{eqb2}
 	-\De u+V(x)u= \big(|x|^{-\mu}*|u|^p\big) |u|^{p-2}u, \mbox{ in }\mb R^n
 \end{equation}
 has been studied extensively.
 In case of $n=3, p=2$ and $\mu=1$, S. Pekar \cite{pekar} describes \eqref{eqb2} in quantum mechanics of a polaron at rest. Under the same assumptions, P. Choquard, in 1976, used \eqref{eqb2}, in modeling of an electron trapped in its own hole, in a certain approximation to Hartree-Fock theory of one component plasma \cite{lieb}.  Buffoni et al. \cite{buffoni} proved existence of at least one non-trivial solution of \eqref{eqb2} when $p=2$, $n= 3$ and potential $V$ is sign changing periodic function with the assumption that $0$ lies in the gap of the spectrum of the operator $-\De+V$.
 In \cite{alves}, Alves et al. obtained the existence of a non-trivial solution via penalization method of the problem
   \begin{equation*}
   	 -\De u+V(x)u =\big( |x|^{-\mu}*F(u)\big)f(u), \mbox{ in }\mb R^n,
   \end{equation*}
 where $0<\mu<n$, $n=3$, $V$ is a continuous real valued function and $F$ is the primitive of $f$. Gao and Yang \cite{gao} studied the following Brezis-Nirenberg type problem
 \begin{align*}
 	-\De u=\la u+\left(\int_{\Om} \frac{|u(y)|^{2^*_\mu}}{|x-y|^\mu}dy\right)|u(x)|^{2^*_{\mu}-2}u(x) \mbox{ in }\Om, \quad u=0 \mbox{ on }\partial\Om,
 \end{align*}
  where $\Om\subset\mb R^n$ is a bounded domain, $n\ge 3$ and $0<\mu<n$. They proved existence, multiplicity and non-existence results with respect to the parameter $\la$.\par
 \noi In the nonlocal case D' Avenia et al. \cite{davenia} considered
 \begin{equation*}
 	(-\De)^s u+\omega u= \big(|x|^{\al-n}*|u|^p \big)|u|^{p-2}u, \mbox{ in }\mb R^n,
 \end{equation*}
 where $\omega>0$, $p>1$ and $s\in(0,1)$. In this work authors proved regularity, existence, non-existence, symmetry as well as decay properties of solution. Mukherjee and Sreenadh \cite{tuhinaBN} obtained Brezis-Nirenberg type results for problem involving fractional Laplacian with Choquard type nonlinearity having critical growth.
%  Giacomoni et al. \cite{giac} obtained existence and multiplicity results for system of equations involving fractional Laplacian with concave growth term and the critical Choquard nonlinearity.
 In \cite{pucci}, Pucci et al. studied problem involving critical Choquard nonlinearity with fractional $p$-Laplacian. For related work on this type of problems, we refer to \cite{goel,mingqi1,moroz,su,xiang} and the references therein.\par
 \noi Partial differential equations involving operator $-\De_p-\De_q$, known as the $(p,q)$-Laplacian, arises from important applications such as biophysics, plasma physics, reaction-diffusion (see \cite{bobkov, cherfils,fife,maranorcn,wilhel}). Due to this, a lot of work has been done in last decade on $(p,q)$-Laplacian problems. Among them,
   Papageorgiou and R\v{a}dulescu \cite{papageo} considered the problem
   \begin{equation*}
      -\De_p u-\De u=f(x,u) \mbox{ in }\Om, u=0 \mbox{ on }\partial\Om,
   \end{equation*}
  where $\Om$ is a bounded domain and $f$ is a Carath\'eodory function. Among other results, authors proved existence of four non-trivial solutions for the case $1<p<2$. Aizicovici et al. \cite{aizic} proved existence of constant sign and nodal solutions for the problem
   \begin{equation*}
  -\De_p u-\mu\De u=f(x,u) \mbox{ in }\Om, u=0 \mbox{ on }\partial\Om,
  \end{equation*}
  where $\Om$ is a bounded domain, $p>2$ and $\mu>0$.
 For general $p$ and $q$, Yin and Yang \cite{yin} considered
  \begin{align*}
  -\Delta_{p}u-\Delta_{q} u=|u|^{p^*-2}u+\theta V(x)|u|^{r-2}u+\la f(x,u) \ \mbox{ in } \Om, \ \ u=0 \ \mbox{ on }\pa\Om,
  \end{align*}
  where  $1<r<q<p<n$ and $f(x,u)$ is a subcritical perturbation and they proved multiplicity of solutions using Lusternik-Schnirelman theory.
   %while Gasi\'nski and Papageorgiou \cite{gasinski} obtained the existence of two positive solutions of the problem with concave nonlinearity and carath\'eodory perturbation having subcritical growth for the case $2\le q\le p<\infty$. Subsequently,
   Marano et. al  \cite{marano} studied the problem with Carath\'eodory function having critical growth. Using critical point theory with truncation arguments and comparison principle authors also proved bifurcation type result. 
   %Bobkov and Tanaka \cite{bobkov} obtained sign changing solutions for $(p,q)$-Laplace equations involving two parameters.  
   \par
 \noi As far as problems involving fractional $(p,q)$-Laplacian is concerned, there is not much literature available. Bhakta and Mukherjee \cite{bhakta} considered the problem
 \begin{equation*}
 (Q_{\theta,\la}) \left\{
 (-\Delta)^{s_1}_{p}u+ (-\Delta)^{s_2}_{q}u = |u|^{p_{s_1}^*-2}u+\theta V(x)|u|^{r-2}u+ \la f(x,u) \; \text{in}\; \Om, \;
 u =0\; \text{in} \ \mb R^n\setminus \Om,
 \right.
 \end{equation*}
 where $0<s_2<s_1<1$, $1<r<q<p<\frac{n}{s_1}$, and  $V$ and $f$ are some appropriate functions. Here  they proved $(Q_{\theta,\la})$ has infinitely many weak solutions for some range of $\la$ and $\theta$. Moreover, for  $V(x)\equiv 1$, $\la=0$, and assuming $r>q$ and certain other conditions on $n$ and $r$, they proved the existence of $cat_\Om(\Om)$ many solutions of $(Q_{\theta,\la})$ using Lusternik-Schnirelmann category theory. Goel et al. \cite{goel1} studied the following fractional $(p,q)$-Laplacian problem
 \begin{equation}\label{eqb4}
 	 (-\Delta)^{s_1}_{p}u+\ba(-\Delta)^{s_2}_{q}u  = \la a(x)|u|^{\delta-2}u+ b(x)|u|^{r-2} u\;  \text{ in }\; \Om, \; u=0 \text{ in } \mathbb{R}^n\setminus \Om,
 \end{equation}
 where $\Om\subset\mb{R}^n$ is a bounded domain, $1<
 \de \le q\leq p<r \leq p^{*}_{s_1}=\ds \frac{np}{n-ps_1}$,
 $0<s_2 < s_1<1$, $n> p s_1$, $\la, \ba>0$, and  $a$ and $b$ are sign changing
 functions. Using Nehari manifold method authors proved existence of at least two non-negative and non-trivial solutions in the subcritical case for all $\ba>0$ and for some range of $\la$. For the critical case under some restriction on $\de$, they obtained multiplicity results in some range of $\ba$ and $\la$. Furthermore, they proved weak solutions of \eqref{eqb4} are in the space $L^\infty(\Om)\cap C^{0,\al}_{\text{loc}}(\Om)$, for some $\al\in(0,1)$, when $2\le q\le p<r<p^*_{s_1}$.\par
 \noi Inspired from all these works we study regularity, existence and multiplicity results of  $(p,2)$ fractional Laplacian problem with critical Choquard type non-linearity. Using Moser iteration technique we prove each weak solution of problem $(P_\la)$ is bounded, in the case $\mu<4s_1$. In this regard, we prove the following theorem.
  \begin{Theorem}\label{thb}
 	Suppose $\mu<4s_1$ and the function $a(x)$ is bounded in $\Om$. Let $u$ be a non-negative solution of problem $(P_\la)$, then $u\in L^\infty(\Om)$.
 \end{Theorem}
 Regarding the existence and multiplicity results,
 using the method of minimization over some suitable subset of the Nehari manifold we obtain existence of at least two non-trivial non-negative solutions for all $\ba>0$ and $\la$ in some range for the subcritical case. In the critical case, we prove the existence of solutions by identifying the first critical level (as defined in Lemma \ref{lem3}), below which the Palais-Smale sequences contain a convergent subsequence. We first prove multiplicity results for all $\ba>0$ and for some range of $\la$ but with some restriction on $q$. For this we estimate the fractional $p$-Laplacian norm of family of minimizers of $S$, the best constant of the embedding of the space $X$ into $L^{2^*_{s_1}}(\Om)$ (see Lemma \ref{L2}). Later, we remove this restriction on $q$ and prove multiplicity result for small $\la$ and $\ba$.
 \noi We show the following existence and multiplicity theorems for problem $(P_\la)$.
\begin{Theorem}\label{th1}
Let $r < 2^{*}_{\mu}$. Then, there exists $\la_0>0$ such that problem $(P_\la)$ has at least two non-negative solutions for all $\la \in(0, \la_0)$ and $\ba>0$.
\end{Theorem}

   \begin{Theorem}\label{th2}
Let $r=2^{*}_{\mu}$ and the function $a(x)$ be continuous in $\Om$. Then, there exist constants $\La,\La_{0}>0$ such that
\begin{enumerate}
	\item[(i)] $(P_\la)$ admits at least one non-negative solution for all $\la \in(0, \La)$ and $\ba>0$,
	\item[(ii)] $(P_\la)$ admits at least two non-negative solutions for all $\la \in(0, \La_{0})$, $\ba>0$ and
	 \begin{enumerate}
	 	\item[(I)] for all $q>0$, provided $1<p<n/(n-s_1)$,
	 	\item[(II)] for all $q\in\big(1, \frac{n(2-p)}{n-2s_1}\big)  \cup\big (\frac{4n}{4n-np-4s_1},p\big)$, provided $n/(n-s_1)\le p<2$.
	 \end{enumerate}
\end{enumerate}
\end{Theorem}
%In part (ii) of the above Theorem the restrictions on $\ba$ can be removed by imposing some restrictions to $q$. To obtain multiplicity result in this regard, we proved some estimate on the fractional $p$-Laplacian norm of the family of minimizers of the best Sobolev constant of the embedding of the space $X$ into $L^{2^*_{s_1}}(\Om)$ (see Lemma \ref{L2}).
 Next, we remove this restriction on $q$, to obtain the existence of second solution in the critical case for all $1<q<p<2$ but for some range of $\la$ and $\ba$. In this regard we state our theorem as follows.
\begin{Theorem}\label{th3}
	Let $r=2^*_\mu$, $1<q<p<2$ and  the function $a(x)$ be continuous in $\Om$. Then, there exist constants $\La,\La_{00},\ba_{00}>0$ such that  $(P_\la)$ has at least two solutions for all $\la \in(0, \La_{00})$ and $\ba\in(0,\ba_{00})$.
\end{Theorem}

We point out that the study of non-autonomous functionals characterized by the fact that the energy density changes its ellipticity and growth properties according to the point has been continued by Mingione {\it et al.} \cite{1Bar-Col-Min, 2Bar-Col-Min, beck}, R\u adulescu {\it et al.} \cite{bahr19, cencelj, prrzamp, prrpams, radopuscula, zhang}, etc. Some of the abstract methods used in this paper can be found in the recent monograph by Papageorgiou, R\u adulescu and Repov\v{s} \cite{prrbook}.

 The paper is organized as follows: In section 2, we provide variational setting and regularity result. In section 3, we define Nehari manifold associated to problem $(P_\la)$ and give fibering map analysis and some preliminary results. In section 4, we prove the existence and multiplicity results in the subcritical case. In section 5, we have the existence and multiplicity of solutions in the critical case.

\section{Variational setting and regularity result}
\setcounter{equation}{0}
Let $\Om$ be any open subset of $\mathbb{R}^n$, consider the function space, which were introduced in \cite{serva} for $p_i=2$ and in \cite{goyal} for general $p$,
  \begin{align*}
    X_{p_i,s_i}:=\bigg\{ u\in L^{p_i}(\mb R^{n}): u=0 \mbox{ a.e. in }\mb R^n\setminus\Om, \int_{Q}\frac{|u(x)-u(y)|^{p_i}}{|x-y|^{n+p_is_i}}~dxdy<\infty   \bigg\},
 \end{align*}
 where $p_1=2$, $p_2=p$, $0<s_i<1$ and $Q= \mathbb{R}^{2n} \setminus (\Om^c\times \Om^c)$, which is a reflexive Banach space when endowed with the norm
 \begin{align}\label{eqb31}
    \|u\|_{X_{p_i,s_i}}:= \left(\int_{Q} \frac{|u(x)-u(y)|^{p_i}}{|x-y|^{n+p_is_i}}~dxdy  \right)^{\frac{1}{p_i}}.
 \end{align}
 Notice that the integral in \eqref{eqb31} can be extended to $\mb{R}^{2n}$ as $u=0$ a.e. on $\mb{R}^n\setminus \Om$.

%
%
%$W^{s,p}(\Om)$ defined as
%\begin{align*}
%W^{s,p}(\Om)= \left\lbrace u \in L^p(\Om): \int_{\Om}\int_{\Om} \frac{|u(x)-u(y)|^p}{|x-y|^{n+sp}}~dxdy <+ \infty \right\rbrace
%\end{align*}
%\noi endowed with the norm
%\begin{align*}
%\|u\|_{W^{s,p}(\Om)}:=  \|u\|_{L^p(\Om)}+\left(\int_{\Om}\int_{\Om} \frac{|u(x)-u(y)|^p}{|x-y|^{n+sp}}~dxdy  \right)^{\frac{1}{p}}.
%\end{align*}
%\noi Consider the space
%\begin{align*}
%\tilde{X}_{m,s_i}:= \bigg\{u:\mathbb{R}^n\ra \mathbb{R} \text{ is measurable }: u \in L^{m}(\Om), \int_{Q} \frac{|u(x)-u(y)|^{m}}{|x-y|^{n+s_im}}~dxdy < \infty  \bigg\},
%\end{align*}
%\noi where $Q= \mathbb{R}^n \setminus (\Om^c\times \Om^c),$ $m\in\{2,p\},\; 0<s_i<1$ and $i=1,2$, then $\tilde{X}_{m,s_i}$ is a Banach space with the norm
%\begin{align*}
%\|u\|_{\tilde{X}_{m,s_i}}:=  \|u\|_{L^m(\Om)}+\left(\int_{Q} \frac{|u(x)-u(y)|^m}{|x-y|^{n+s_im}}~dxdy  \right)^{\frac{1}{m}}.
%\end{align*}
%\noi Denote  $X_{m,s_i}$ the closure of $C_c^{\infty}(\Om)$ in $\tilde{X}_{m,s_i}$, then it can be shown that $X_{m,s_i}$ is a reflexive Banach space with the norm
%\begin{align}\label{eqb31}
%\|u\|_{X_{m,s_i}}:=  \left(\int_{Q} \frac{|u(x)-u(y)|^m}{|x-y|^{n+s_im}}~dxdy  \right)^{\frac{1}{m}}.
%\end{align}
%\noi Notice that the integral in \eqref{eqb31} can be extended to $\mathbb{R}^{2n}$.

For simplicity, we denote $X_1:=X_{2,s_1}$ and $X_2:=X_{p,s_2}$ and corresponding norms by $\|\cdot\|_{X_1}$ and $\|\cdot\|_{X_2}$, respectively. From \cite{serva}, we have the continuous embedding of $X_{1}$ into $L^m(\Om)$ for $1\le m\le 2^*_{s_1}$, therefore we define
\begin{equation*}
     S_m=\inf_{u\in X_{1}\setminus\{0\}} \frac{\|u\|_{X_{1}}^{2}}{\|u\|_{m}^{2}}.
\end{equation*}
For the sake of convenience, we denote $S_{2^*_{s_1}} =S$.\\
Regarding the spaces $X_1$ and $X_2$, we have the following relation.
\begin{Lemma}\label{L1}
	Let $1<p\leq 2$ and $0<s_2<s_1<1$, then there exists a constant $C=C(|\Om|,n,p,s_1,s_2)>0$ such that
	\begin{align*}
	\|u\|_{X_2}\leq C \|u\|_{X_1}, \; \;  \forall \; u \in X_1.
	\end{align*}
\end{Lemma}
\begin{proof}
	Proof follows from \cite[Lemma 2.1]{goel1}.\QED
	\end{proof}
The nonlocal nonlinear Choquard term present in the right hand side of $(P_\la)$ is well defined due to the following result.
 \begin{Theorem}(Hardy-Littlewood-Sobolev inequality)\label{HLS}
	Let $t,r>1$ and $0<\mu <n$ with $1/t+\mu/n+1/r=2$, $f\in L^t(\mathbb{R}^n)$ and $h\in L^r(\mathbb{R}^n)$. There exists a sharp constant $C(t,r,\mu,n)>0$ independent of $f,h$, such that
	\begin{align*}
	   \int_{\mathbb{R}^n}\int_{\mathbb{R}^n}\frac{f(x)h(y)}{|x-y|^{\mu}}~dxdy \leq C(t,r,\mu,n) |f|_t|h|_r.
	\end{align*}
\end{Theorem}
In general, let $f = h= |u|^r$, then by Hardy-Littlewood-Sobolev inequality, we get
\[ \int_{\mb R^n}\int_{\mb R^n} \frac{|u(x)|^r|u(y)|^r}{|x-y|^{\mu}}~dxdy\le C(n,\mu,r)\|u\|^{2r}_{tr},\]
 if $|u|^r \in L^t(\mb R^n)$ for some $t>1$ satisfying
$\frac{2}{t}+ \frac{\mu}{n}=2.$
Thus, for $u \in H^{s_1}(\mb R^n)$, by Sobolev embedding theorems, we must have
\[ \frac{2n-\mu}{n}\leq r \leq \frac{2n-\mu}{n-2s_1}.\]
The term $2^*_\mu := (2n-\mu)/(n-2s_1)$ is known as the upper critical exponent in the sense of Hardy-Littlewood-Sobolev inequality. In particular, when $r=  \frac{2n-\mu}{n-2s_1}$,  for $u\in X_1$, we have
\[\left(\int_{\mb R^n} \int_{\mb R^n}\frac{|u(x)|^{2^*_{\mu}}|u(y)|^{2^*_{\mu}}}{|x-y|^{\mu}}~dxdy\right)^{1/2^*_{\mu}} \leq C(n,\mu)^{\frac{1}{2^*_{\mu}}}\|u\|^2_{2^*_{s_1}},\]
where $C(n,\mu)$ is a suitable constant and $2^*_{s_1} = 2n/(n-2s_1)$. Let us define
\begin{align*}
	S_H:=\inf_{u\in H^{s_1}(\mb R^n)\setminus\{0\}} \frac{\int_{\mb R^n}\int_{\mb R^n}\frac{|u(x)-u(y)|^2}{|x-y|^{n+2s_1}}dxdy}{\left(\int_{\mb R^n}\int_{\mb R^n}\frac{|u(x)|^{2^*_\mu}|u(y)|^{2^*_\mu}}{|x-y|^{\mu}}dxdy\right)^\frac{1}{2^*_\mu}},
\end{align*}
which is achieved if and only if $u$ is of the form
  \begin{align*}
  	C\left(\frac{t}{t^2+|x-x_0|^2}\right)^{(n-2s_1)/2}, \mbox{ for }x\in\mb R^n,
  \end{align*}
  for some $x_0\in\mb R^n$, $C>0$ and $t>0$ (see \cite{tuhinaBN}). Moreover,
  \begin{equation}\label{eqb9}
  S_H=\frac{S}{\big(C(n,\mu)\big)^{1/2^*_\mu}}.
  \end{equation}

\begin{Lemma}(\cite[Lemma 2.2]{tuhinaBN})
	Define
	\begin{align*}
		S_H(\Om):=\inf_{u\in X_1\setminus\{0\}} \frac{\int_{Q}\frac{|u(x)-u(y)|^2}{|x-y|^{n+2s_1}}dxdy}{\left(\int_{\Om}\int_{\Om}\frac{|u(x)|^{2^*_\mu}|u(y)|^{2^*_\mu}}{|x-y|^{\mu}}dxdy\right)^{1/2^*_\mu}}.
	\end{align*}
	Then $S_H=S_H(\Om)$ and $S_H(\Om)$ is never achieved except when $\Om=\mb R^n$.
\end{Lemma}

\begin{Lemma}(\cite[Lemma 3.4]{tuhinaBN})
	Let
	\begin{align*}
       \|u\|_{NL}=\left(\int_{\Om}\int_{\Om}\frac{|u(x)|^{2^*_\mu}|u(y)|^{2^*_\mu}}{|x-y|^{\mu}}dxdy\right)^{1/2.2^*_\mu},
    \end{align*}
    then $\|\cdot\|_{NL}$ defines a norm on $Y:=\{u:\Om\to\Om: \ u \text{ is measurable, } \|u\|_{NL}<\infty\}$.
\end{Lemma}

\noi\textbf{Notations:} For simplicity, for $p_1=2$ and $p_2=p$, we will use the following notations
\begin{align*}
&\mc K_i(u,v)=\ds \int_{Q}\frac{|u(x)-u(y)|^{p_i-2}(u(x)-u(y))(v(x)-v(y))}{|x-y|^{n+p_i s_i}}~ dxdy,\;\text{ for all } u,v \in X_{i},\\
&\mc A_r(u,v)= \int_{\Om} \int_{\Om} \frac{ |u(y)|^{r}|u(x)|^{r-2}u(x)v(x)}{|x-y|^{\mu}} ~dx dy \;\; \text{ for all } u,v \in X_{1}
\end{align*}
and $\hat{r}=\frac{2n}{2n-\mu}r$.
\begin{Definition}
	A function $u \in X_1$ is said to be  a solution of problem $(P_\la)$, if for all $v\in X_1$
	\begin{align*}
	 \mc{K}_1(u,v)+\ba\mc K_2(u,v)
	 - \la \int_{\Om}a(x)|u|^{q-2}uv~dx
	- \mc A_r(u,v)=0.
		\end{align*}
\end{Definition}
The Euler functional $\mc{I}_{\la}: X_1 \ra\mb R$ associated to the problem $(P_\la)$ is defined as
\begin{align*}
\mc{I}_{\la}(u)=\frac{1}{2} \|u\|_{X_1}^2+\frac{\ba}{p} \|u\|_{X_2}^p-\frac{\la}{q}\int_{\Om}a(x)|u|^{q} dx
-\frac{1}{2r}\int_{\Om} \int_{\Om} \frac{ |u(x)|^{r}|u(y)|^r}{|x-y|^{\mu}} ~dx dy.
\end{align*}

\noi Now we present a regularity result, namely the $L^\infty$ bound, for weak solutions of problem $(P_\la)$.\\
\noi\textbf{Proof of Theorem \ref{thb}}: Let $u$ be a nonnegative solution of $(P_\la)$.
	For $\vartheta>1$ and $T>0$, define
	\begin{align*}
	\phi(t)=\phi_{T,\vartheta}(t)=\begin{cases}
	0,\quad \mbox{ if }t\le 0 \\
	t^\vartheta, \quad \mbox{ if }0<t<T \\
	\vartheta T^{\vartheta-1}t-(\vartheta-1)T^\vartheta, \quad \mbox{ if } t\ge T.
	\end{cases}
	\end{align*}
	Then,
	\begin{align*}
	\phi^\prime(t)=\begin{cases}
	0,\quad \mbox{ if }t\le 0 \\
	\vartheta t^{\vartheta-1}, \quad \mbox{ if }0<t<T \\
	\vartheta T^{\vartheta-1}, \quad \mbox{ if } t\ge T.
	\end{cases}
	\end{align*}
	As $u(x)\ge 0$, it is easy to observe that
	$$ \phi(u)\ge 0, \; \phi^\prime(u)\ge 0 \mbox{ and } u\phi^\prime(u)\le\vartheta\phi(u).$$
	Since $\phi$ is Lipschitz, we have $\phi(u)\in X_1$ and
	\begin{align*}
	\|\phi(u)\|_{X_1}=\left(\int_Q\frac{|\phi(u(x))-\phi(u(y))|^2}{|x-y|^{n+2s_1}}~dxdy\right)^{1/2}\leq K\|u\|_{X_1}.
	\end{align*}
	Moreover, we have
	\begin{equation}\label{eqb41}
	S\|\phi(u)\|_{2^*_{s_1}}^2\le\|\phi(u)\|_{X_1}^2=\int_{\Om}\phi(u)(-\De)^{s_1}\phi(u).
	\end{equation}
	Since $\phi$ is convex, we observe that $\phi(u)\phi^\prime(u)\in X_1$ and
	\begin{equation}\label{eqb42}
	\int_{\Om}\phi(u)(-\De)^{s_1}\phi(u)\le \int_{\Om}\phi(u)\phi^\prime(u)(-\De)^{s_1}u.
	\end{equation}
	Define $g:\mb R\to\mb R$ as $g(t)=\phi(t)\phi^\prime(t)$, then $g$ is increasing, and set $G(t)=\int_{0}^{t}g^\prime(s)^{\frac{1}{p}}ds$ for $t\in\mb R$. Using the inequality (see \cite[Lemma A.2]{brasegn})
	\begin{align*}
	|a-b|^{p-2}(a-b)\big(g(a)-g(b)\big)\ge |G(a)-G(b)|^p, \mbox{ for } a,b\in\mb R,
	\end{align*}
	we obtain
	\begin{equation}\label{eqb44}
	\begin{aligned}
	&\int_{Q}\frac{|u(x)-u(y)|^{p-2}\big(u(x)-u(y)\big)\big(\phi(u(x))\phi^\prime(u(x))-\phi(u(y))\phi^\prime(u(y))\big)}{|x-y|^{n+ps_2}}dxdy\\
	&\ge\int_{Q}\frac{|G(u(x))-G(u(y))|^p}{|x-y|^{n+ps_2}}dxdy\ge 0.
	\end{aligned}
	\end{equation}
	Noting the fact that $1<q<2$, that is, $0<q-1<1$, there exists a constant $C>0$ such that $\la |a(x)|u|^{q-2}u|\le C(1+|u|)$. Thus, from \eqref{eqb41}, \eqref{eqb42} and \eqref{eqb44}, we deduce that
	\begin{align*}
	S\|\phi(u)\|_{2^*_{s_1}}^2&\le\|\phi(u)\|_{X_1}^2= \int_{\Om}\phi(u)(-\De)^{s_1}\phi(u)\\
	&\le \int_{\Om}\phi(u)\phi^\prime(u)(-\De)^{s_1}u \\ &\quad+\int_{Q}\frac{|u(x)-u(y)|^{p-2}\big(u(x)-u(y)\big)\big(\phi(u(x))\phi^\prime(u(x))-\phi(u(y))\phi^\prime(u(y))\big)}{|x-y|^{n+ps_2}}dxdy\\
	&=\la\int_{\Om}a(x)|u|^{q-2}u\phi(u)\phi^\prime(u)dx+\int_{\Om} \int_{\Om} \frac{|u(y)|^{r}|u(x)|^{r-2}u(x)\phi(u)\phi^\prime(u)}{|x-y|^{\mu}}dxdy\\
	&\le C\int_{\Om}(1+u)\phi(u)\phi^\prime(u)+\int_{\Om} \int_{\Om} \frac{|u(y)|^{r}|u(x)|^{r-2}u(x)\phi(u)\phi^\prime(u)}{|x-y|^{\mu}}dxdy.
	\end{align*}
	Using the relation $u\phi^\prime(u)\le\vartheta\phi(u)$ and $\phi^\prime(u)\le\vartheta(1+\phi(u))$, we get
	\begin{equation}\label{eqb45}
	S\|\phi(u)\|_{2^*_{s_1}}^2\le C\vartheta\int_{\Om}\big(\phi(u)^2+\phi(u)\big)+\vartheta\int_{\Om} \int_{\Om} \frac{|u(y)|^{r}|u(x)|^{r-2}\phi(u)^2}{|x-y|^{\mu}}dxdy.
	\end{equation}
	Now, using H\"older inequality, we see that
	\begin{equation}\label{eqb50}
	\int_{\Om}\phi(u)\le\left(\int_{\Om}\phi(u)^2\right)^{1/2}|\Om|^{1/2}\le \frac{1}{2}\int_{\Om}\phi(u)^2+\frac{1}{2}|\Om|.
	\end{equation}
	Next, to estimate the second term in \eqref{eqb45}, we follow the approach similar to \cite[Proof of Theorem 1(2)]{su}. Using Theorem \ref{HLS}, we have
	\begin{align}\label{eqb52}
	\int_{\Om} \int_{\Om} \frac{|u(y)|^{r}|u(x)|^{r-2}\phi(u)^2}{|x-y|^{\mu}}dxdy\le C(n,\mu) \||u|^r\|_{\frac{2n}{2n-\mu}} \||u|^{r-2}\phi(u)^2\|_{\frac{2n}{2n-\mu}}.
	\end{align}
	By the embedding results of $X_1$, we get $u\in L^{2^*_{s_1}}$, so we can assume $\|u\|_{2^*_{s_1}}\leq C$. Employing H\"older inequality, we get
%	\begin{align*}
%		\||u|^r\|_{\frac{2n}{2n-\mu}}\leq \left( \left(\int_{\Om}|u|^{2^*_{s_1}}\right)^\frac{(n-2s_1)}{(2n-\mu)r} |\Om|^{1-\frac{n-2s_1}{(2n-\mu)r}} \right)^{(2n-\mu)/2n},
%	\end{align*}
%	which implies
	$	\||u|^r\|_{\frac{2n}{2n-\mu}}\leq C$. Next, for $m>0$, we deduce that
	\begin{align*}
		\left(\int_{\Om} \big(|u|^{r-2}\phi(u)^2\big)^\frac{2n}{2n-\mu}  \right)^\frac{2n-\mu}{2n}&=\bigg[ \left(\int_{\{u<m\}}+\int_{\{u\ge m\}}\right)\big(|u|^{r-2}\phi(u)^2\big)^\frac{2n}{2n-\mu}\bigg]^\frac{2n-\mu}{2n}\\
		 %&\le\left(\int_{\{u>m\}} \big(|u|^{r-2}\phi(u)^2\big)^\frac{2n}{2n-\mu}  \right)^\frac{2n-\mu}{2n} \\ &\qquad +\left(\int_{\{u\ge m\}} \big(|u|^{r-2}\phi(u)^2\big)^\frac{2n}{2n-\mu}  \right)^\frac{2n-\mu}{2n}\\
		 &\le m^{r-2} \left(\int_{\{u<m\}} \big(\phi(u)\big)^\frac{2.2n}{2n-\mu}  \right)^\frac{2n-\mu}{2n}\\
		 &\qquad+\left(\int_{\{u\ge m\}} \big(|u|^{r-2}\phi(u)^2\big)^\frac{2n}{2n-\mu}  \right)^\frac{2n-\mu}{2n}.
	\end{align*}
  With the help of H\"older inequality, we obtain
   \begin{align*}
   	 \left(\int_{\{u\ge m\}} \big(|u|^{r-2}\phi(u)^2\big)^\frac{2n}{2n-\mu}  \right)^\frac{2n-\mu}{2n}\leq \left(\int_{\Om}\phi(u)^{2^*_{s_1}} \right)^\frac{2}{2^*_{s_1}} \left(\int_{\{u\ge m\}}|u|^{\frac{(r-2)}{2^*_\mu-2}2^*_{s_1}} \right)^{\frac{4s_1-\mu}{2n}}.
   \end{align*}
   By the fact $r\le 2^*_{\mu}$ and $\|u\|_{2^*_{s_1}}\le C$, we can find $m>0$ such that
     \begin{align*}
     	\left(\int_{\{u\ge m\}}|u|^{\frac{(r-2)}{2^*_\mu-2}2^*_{s_1}} \right)^{\frac{4s_1-\mu}{2n}}\leq \frac{S}{2.\vartheta C}.
     \end{align*}
     Therefore, collecting these informations in \eqref{eqb52}, we get
     \begin{equation}\label{eqb51}
      \begin{aligned}
     		\int_{\Om} \int_{\Om} \frac{|u(y)|^{r}|u(x)|^{r-2}\phi(u)^2}{|x-y|^{\mu}}dxdy\le  C&\left( m^{r-2} \left(\int_{\Om} \big(\phi(u)\big)^\frac{4n}{2n-\mu}  \right)^\frac{2n-\mu}{2n} \right.\\  &\qquad \left.+ \frac{S}{2.\vartheta C} \left(\int_{\Om}\phi(u)^{2^*_{s_1}} \right)^\frac{2}{2^*_{s_1}}\right).
      \end{aligned}
     \end{equation}
  Hence, using \eqref{eqb50} and \eqref{eqb51} in \eqref{eqb45}, we obtain
	\begin{align*}
	\|\phi(u)\|_{2^*_{s_1}}^2\le 2.\vartheta C\left(1+\int_{\Om}\phi(u)^2+\left(\int_{\Om}|\phi(u)|^\frac{4n}{2n-\mu}\right)^\frac{2n-\mu}{2n}\right).
	\end{align*}
	Taking $T\ra\infty$, we get
	\begin{align*}
	\left(\int_{\Om}u^{2^*_{s_1}\vartheta}\right)^\frac{2}{2^*_{s_1}}\le C\vartheta\left(1+\int_{\Om}u^{2\vartheta}+\left(\int_{\Om}u^{2\vartheta\frac{2^*_{s_1}}{2^*_{\mu}}}\right)^\frac{2^*_{\mu}}{2^*_{s_1}}\right),
	\end{align*}
	with the help of H\"older inequality, we deduce that
	\begin{align*}
	\left(\int_{\Om}u^{2^*_{s_1}\vartheta}\right)^\frac{2}{2^*_{s_1}}\le C\vartheta\Bigg[1+(1+|\Om|^{1-\frac{2^*_{\mu}}{2^*_{s_1}}}) \left(\int_{\Om}u^{2\vartheta\frac{2^*_{s_1}}{2^*_{\mu}}}\right)^\frac{2^*_{\mu}}{2^*_{s_1}}\Bigg]\le C\vartheta\Bigg[1+\left(\int_{\Om}u^{2\vartheta\frac{2^*_{s_1}}{2^*_{\mu}}}\right)^\frac{2^*_{\mu}}{2^*_{s_1}}\Bigg],
	\end{align*}
	therefore,
	\begin{equation}\label{eqb46}
	\left(\int_{\Om}u^{2^*_{s_1}\vartheta}\right)^{1/2^*_{s_1}\vartheta}\le (C\vartheta)^\frac{1}{2\vartheta}
	\left(1+\left(\int_{\Om}u^{2\vartheta\frac{2^*_{s_1}}{2^*_{\mu}}}\right)^{2^*_{\mu}/2^*_{s_1}}\right)^{1/2\vartheta}.
	\end{equation}
	Now we consider the following cases:\\
	Suppose there exists a sequence $\vartheta_k\ra\infty$ such that
	$ \ds\int_\Om u^{2\vartheta_k\frac{2^*_{s_1}}{2^*_{\mu}}}\le 1,$ which implies
	\begin{align*}
	\left(1+\left(\int_{\Om}u^{2\vartheta_k\frac{2^*_{s_1}}{2^*_{\mu}}}\right)^{2^*_{\mu}/2^*_{s_1}}\right)^{1/2\vartheta_k}\le 2^\frac{1}{2\vartheta_k},
	\end{align*}
	thus, from \eqref{eqb46}, we get $u\in L^\infty(\Om)$. Otherwise, there exists $\vartheta_0>0$ such that
	\begin{align*}
	\int_\Om u^{2\vartheta\frac{2^*_{s_1}}{2^*_{\mu}}}> 1, \mbox{ for all } \vartheta\ge\vartheta_0.
	\end{align*}
	Therefore, from \eqref{eqb46}, we obtain
	\begin{equation}\label{eqb47}
	\|u\|_{2^*_{s_1}\vartheta}\le (C\vartheta)^\frac{1}{2\vartheta}\|u\|_{2\vartheta\frac{2^*_{s_1}}{2^*_{\mu}}}
	\quad \mbox{for all }\vartheta\ge\vartheta_0.
	\end{equation}
	Define a sequence $\{\vartheta_k\}$ such that $2^*_{s_1}\vartheta_k=2\vartheta_{k+1}\frac{2^*_{s_1}}{2^*_{\mu}}$, that is, $2^*_{\mu}\vartheta_k=2\vartheta_{k+1}$ for $k\ge 0$. By means of \eqref{eqb46}, it is easy to notice  that $\|u\|_{2^*_{s_1}\vartheta_0}<\infty$. Thus, from \eqref{eqb47}, we get
	\begin{align*}
	\|u\|_{2^*_{s_1}\vartheta_1}\le (C\vartheta_1)^\frac{1}{2\vartheta_1}\|u\|_{2\vartheta_1\frac{2^*_{s_1}}{2^*_{\mu}}} =C^\frac{1}{2\vartheta_1}\vartheta_1^\frac{1}{2\vartheta_1}\|u\|_{2^*_{s_1}\vartheta_0}<\infty,
	\end{align*}
	and
	\begin{align*}
	\|u\|_{2^*_{s_1}\vartheta_2}\le C^{\frac{1}{2\vartheta_1}+\frac{1}{2\vartheta_2}}\vartheta_1^\frac{1}{2\vartheta_1} \vartheta_2^\frac{1}{2\vartheta_2}\|u\|_{2^*_{s_1}\vartheta_0},
	\end{align*}
	similarly we obtain
	\begin{equation}\label{eqb48}
	\|u\|_{2^*_{s_1}\vartheta_k}\le C^{\sum_{m=1}^{k}\frac{1}{2\vartheta_m}}\prod_{m=1}^k(\vartheta_m^\frac{1}{2\vartheta_m}) \|u\|_{2^*_{s_1}\vartheta_0}.
	\end{equation}
	By ratio test we can see that $\ds\sum_{m\in\mb N}\frac{1}{2\vartheta_m}<\infty$. Let $z_k=\prod_{m=1}^k\vartheta_m^\frac{1}{2\vartheta_m}$. Again by using ratio test, we get $\ln z_k=\sum_{m=1}^{k}\frac{\ln \vartheta_m}{2\vartheta_m}$ is convergent. Hence there exists a positive constant $A$ such that  $C^{\sum_{m=1}^{k}\frac{1}{2\vartheta_m}}\prod_{m=1}^k(\vartheta_m^\frac{1}{2\vartheta_m})\le A$ for all $k\ge 0$. Therefore, \eqref{eqb48} implies
	\begin{equation}\label{eqb49}
	\|u\|_{2^*_{s_1}\vartheta_k}\le A \|u\|_{2^*_{s_1}\vartheta_0}<\infty.
	\end{equation}
	We claim that $u\in L^\infty(\Om)$. Suppose not, then there exists $\varepsilon>0$ and $M\subset\Om$ with $|M|>0$ such that
	\begin{align*}
	u(x)\ge A\|u\|_{2^*_{s_1}\vartheta_0}+\varepsilon \quad \mbox{a.e. }x\in M.
	\end{align*}
	Thus,
	\begin{align*}
	\|u\|_{2^*_{s_1}\vartheta_k}\ge\left(\int_M |u(x)|^{2^*_{s_1}\vartheta_k}\right)^\frac{1}{2^*_{s_1}\vartheta_k}\ge (A\|u\|_{2^*_{s_1}\vartheta_0}+\varepsilon) |M|^\frac{1}{2^*_{s_1}\vartheta_k},
	\end{align*}
	this implies that
	\begin{align*}
	\liminf_{k\ra\infty}\|u\|_{2^*_{s_1}\vartheta_k}\ge (A\|u\|_{2^*_{s_1}\vartheta_0}+\varepsilon),
	\end{align*}
	which is a contradiction to \eqref{eqb49}. Hence, $u\in L^\infty(\Om)$.\QED

\section{Nehari manifold and fibering map analysis }
\noi Due to the fact that $1< r$, we see that $\mc I_\la(tu)\ra -\infty$, as $t\ra\infty$ for $u(\not\equiv 0)\in X_1$. Hence, the functional $\mc I_\la$ is not bounded below on $X_1$. Therefore, it is necessary to restrict $\mc I_\la$ to a proper subset of $X_1$ on which it is bounded below. For this reason, we consider
the Nehari set $\mc{M}_\la$ associated to $(P_\la)$, which is defined as
\begin{equation*}
  \mathcal{M}_\la=\{{u\in X_1\setminus\{0\}}: \langle \mathcal{I}_\lambda^\prime(u), u\rangle =0\},
\end{equation*}
where $\langle\; ,\; \rangle$ is the duality between $X_1$ and its dual space. Obviously,  $\mc M_{\la}$ contains all the solution of $(P_\la)$. To study the critical points of the functional $\mc I_\la$, we define the fibering maps associated to it.
 %Consider the fibering map for the functional $\mathcal{I}_{\la}$ which were introduced by Drabek and Pohozaev in \cite{DP}.
 For $u \in X_1$,  define $\varphi_u: \mb R^+\ra \mb R$
  as $\varphi_{u}(t)=\mathcal{I}_{\la}(tu)$, that is
  \begin{align}\nonumber
  \varphi_{u}(t) &= \frac{t^2}{2} \|u\|_{X_1}^2+\ba \frac{t^p}{p} \|u\|_{X_2}^p- \frac{\la
  	t^{q}}{q}\int_{\Om} a(x)|u|^{q} dx - \frac{t^{2r}}{2r}\int_{\Om} \int_{\Om} \frac{ |u(x)|^{r}|u(y)|^r}{|x-y|^{\mu}} ~dx dy ,\\\label{eqb13}
  \varphi_{u}^{\prime}(t) &= t \|u\|_{X_1}^2+\ba t^{p-1} \|u\|_{X_2}^p- {\la
  	t^{q-1}}\int_{\Om} a(x) |u|^{q} dx  - t^{2r-1}\int_{\Om} \int_{\Om} \frac{ |u(x)|^{r}|u(y)|^r}{|x-y|^{\mu}} ~dx dy,\\\label{eqb14}
  \varphi_{u}^{\prime\prime}(t) &= \|u\|_{X_1}^2+\ba (p-1)t^{p-2} \|u\|_{X_2}^p- (q-1) \la
  t^{q-2} \int_{\Om} a(x) |u|^{q} dx \nonumber\\
  &\qquad - (2r-1) t^{2r-2}\int_{\Om} \int_{\Om} \frac{ |u(x)|^{r}|u(y)|^r}{|x-y|^{\mu}}~dx dy.
  \end{align}
  It is clear that $tu\in \mathcal{M}_{\la}$ if and only if
  $\varphi_{u}^{\prime}(t)=0$ and in particular, $u\in \mathcal{M}_{\la}$ if
  and only if $\varphi_{u}^{\prime}(1)=0$. Hence, it is natural to split
  $\mathcal{M}_{\la}$ into three parts corresponding to local minima,
  local maxima and points of inflection, namely
  \begin{align*}
  \mathcal{M}_{\la}^{0}:= \left\{u\in \mathcal{M}_{\la}:
  \varphi_{u}^{\prime\prime}(1) = 0\right\}, \text{ and }\mathcal{M}_{\la}^{\pm}&:= \left\{u\in \mathcal{M}_{\la}:
  \varphi_{u}^{\prime\prime}(1)
  \gtrless 0\right\}.
  \end{align*}
%\begin{Lemma}\label{lmcp}
%If $u$ is a local minimum or maximum  of $\mathcal{I}_{\lambda}$ on $\mathcal{M}_{\lambda}$  and $u \notin \mathcal{M}_{\lambda}^{0}.$ Then $u$ is a critical point for $\mathcal{I}_{\lambda}.$
%\end{Lemma}
%%\begin{proof} \cite{dpp}.
%%\QED
%%\end{proof}

\noindent Define $\sigma_{\lambda} := \inf\{\; \mathcal{I}_{\lambda}(u) \;| \;u \in \mathcal{M}_{\lambda}\}$ and $\sigma_{\lambda}^\pm := \inf\{\; \mathcal{I}_{\lambda}(u) \;| \;u \in \mathcal{M}_{\lambda}^\pm\}$.
%$\sigma_{\lambda}^- := \inf\{\mathcal{I}_{\lambda}(u)| u \in \mathcal{M}_{\lambda}^-\}$ and $\sigma_{\lambda}^+ := \inf\{\mathcal{I}_{\lambda}(u)| u \in \mathcal{M}_{\lambda}^+\}$.
\begin{Lemma}\label{le44}
$\mathcal{I}_{\lambda}$ is coercive and bounded below on  $\mathcal{M}_{\lambda}$.
%Moreover, there exists a constant $C = C(p,q,r)>0$ such that $\mathcal{I}_{\lambda} > - C \lambda^{r/(r-q)}.$
\end{Lemma}
\begin{proof}
Proof follows using H\"older inequality and Sobolev embedding results.
% we have
%\begin{align*}
%\mathcal{I}_{\lambda}(u) &= \left(\frac{1}{2}-\frac{1}{2r}\right)\|u\|_{X_1}^{2}+\ba\left(\frac{1}{p}-\frac{1}{2r}\right)\|u\|_{X_2}^{p}-\la\left(\frac{1}{q}-\frac{1}{2r}\right)\int_\Omega a(x)|u|^{q}~dx,\\
% &\geq \left(\frac{1}{2}-\frac{1}{2r}\right)\|u\|_{X_1}^{2}-\lambda \left(\frac{1}{q}-\frac{1}{2r}\right)\|a\|_{\frac{d}{d-q}}S_d^{\frac{-q}{2}} \|u\|_{X_1}^{q}
%\end{align*}
%Thus $\mathcal{I}_\lambda$ is coercive and bounded below in $\mathcal{M}_\lambda$.
\QED
\end{proof}
\begin{Lemma}\label{b5}
There exists $\lambda_{0} > 0$ such that for all $\la \in (0, \la_{0})$, we have $\mathcal{M}_{\la}^{0} = \emptyset.$
\end{Lemma}
\begin{proof}
We distinguish the following cases:\\
 \textbf{Case 1:}  $u \in \mathcal{M}_{\lambda}$ such that $ \ds \int_{\Om}a(x) |u|^{q}~dx = 0.$\\
  From \eqref{eqb13}, we have
  \begin{align*}
 \|u\|_{X_1}^2 +\ba\|u\|_{X_2}^p- \int_{\Om} \int_{\Om} \frac{ |u(x)|^{r}|u(y)|^r}{|x-y|^{\mu}} ~dx dy =0.
  \end{align*}
  Therefore,
  \begin{align*}
   \phi_u^{\prime\prime}(1)= \|u\|_{X_1}^{2}+ \ba (p-1)\|u\|_{X_2}^{p}-&(2r-1) \int_{\Om} \int_{\Om} \frac{ |u(x)|^{r}|u(y)|^r}{|x-y|^{\mu}} ~dx dy \\
  &=(2-2r)\|u\|_{X_1}^{2}+ \ba (p-2r)\|u\|_{X_2}^{p}<0,
  \end{align*}
 \noi which implies $u \notin \mathcal{M}_{\lambda}^{0}$.\\
   \textbf{Case 2:} $u \in \mathcal{M}_{\lambda}$ such that  $ \ds\int_{\Om}a(x) |u|^{q}~dx \neq 0.$\\
If $u \in \mc{M}_{\la}^{0}$, then from \eqref{eqb13} and \eqref{eqb14}, we have
    \begin{align}
      &(2-q)\|u\|_{X_1}^2 +\ba(p-q) \|u\|_{X_2}^p
        = (2r-q)\int_{\Om} \int_{\Om} \frac{ |u(x)|^{r}|u(y)|^r}{|x-y|^{\mu}} ~dx dy \; \; \text{ and }\label{eqb15} \\
    & (2r-2) \|u\|_{X_1}^2+\ba (2r-p)\|u\|_{X_2}^p=  \la(2r-q) \int_\Omega a(x)|u|^{q} dx.\label{eqb22}
    \end{align}
    Define $E_{\la}:\mc{M}_{\la} \rightarrow \mb{R}$ as
  \begin{equation*}
    E_{\lambda}(u) = \frac{(2r-2) \|u\|_{X_1}^2+\ba (2r-p)\|u\|_{X_2}^p}{(2r-q)} - \lambda\int_\Omega a(x)|u|^{q}~dx,
  \end{equation*}
 then from \eqref{eqb22}, $E_{\lambda}(u) = 0$ for all  $u\; \in \mathcal{M}_{\lambda}^{0}.$  Additionally, using H\"older inequality, we have
  \begin{equation}\label{2.13}
  \begin{aligned}
  E_{\la}(u)
 % &\geq  \left(\frac{2r-2}{2r-q}\right)\|u\|_{X_1}^{2} - \la \int_\Om a(x)|u|^{q} dx\\
  %& \geq  \left(\frac{2r-2}{2r-q}\right)\|u\|_{X_1}^{2} - \la {\|a\|_{\frac{d}{d-q}}}S_d^{\frac{-q}{2}}\|u\|_{X_1}^{q},\\
  & \geq \|u\|_{X_1}^{q}\left[\left(\frac{2r-2}{2r-q}\right)\|u\|_{X_1}^{(2-q)} - \la {\|a\|_{\frac{d}{d-q}}}S_d^{\frac{-q}{2}}\right],
  \end{aligned}
  \end{equation}
    Now from \eqref{eqb15} and Theorem \ref{HLS}, we get
  \begin{equation*}
  \|u\|_{X_1} \geq \left(\frac{(2-q)S^r_{\hat r}} {(2r-q)C(n,\mu)}\right)^{\frac{1}{2r-2}}.
  \end{equation*}
   Using this in \eqref{2.13}, we obtain
\begin{equation*}
E_{\la}(u) \geq \|u\|_{X_1}^{q}\left(\left(\frac{2r-2}{2r-q}\right)\left(\frac{(2-q)S^r_{\hat r}}{(2r-q)C(n,\mu)}\right)^{\frac{2-q}{2r-2}} - \la {\|a\|_{\frac{d}{d-q}}}S_d^{\frac{-q}{2}}\right).  \end{equation*}
  Set
\begin{equation}\label{eqb10}
  \lambda_0:=\;\left(\frac{(2r-2)S_d^{\frac{q}{2}}}{(2r-q)\|a\|_{\frac{d}{d-q}}}\right)\left(\frac{(2-q)S^r_{\hat r}}{(2r-q)C(n,\mu)}\right)^{\frac{2-q}{2r-2}}>0,
\end{equation}
then from \eqref{2.13},  for $\lambda\in (0, \lambda_0)$ we get $E_{\lambda}(u)>0,$ for all  $u \in \mathcal{M}_{\lambda}^{0},$
  which is a contradiction. Therefore, $\mathcal{M}_{\lambda}^{0}= \emptyset$ for all $\la\in(0,\la_0)$.
\QED
\end{proof}

\noi Now, define $\psi_u: \mb R^{+} \lra \mb R$ by
\begin{equation*}
 \psi_u(t)= t^{2-q}\|u\|_{X_1}^{2}+\ba t^{p-q}\|u\|_{X_2}^{p} - t^{2r-q}\int_{\Om} \int_{\Om} \frac{ |u(x)|^{r}|u(y)|^r}{|x-y|^{\mu}} ~dx dy,
 \end{equation*}
 then
 \begin{equation*}
 \psi_u^{\prime}(t)=(2-q) t^{1-q} \|u\|_{X_1}^{2}+\ba(p-q) t^{p-q-1} \|u\|_{X_2}^{p} -
(2r-q)t^{2r-q-1}\int_{\Om} \int_{\Om} \frac{ |u(x)|^{r}|u(y)|^r}{|x-y|^{\mu}}~dx dy.
\end{equation*}
\noi Then trivially,  $tu\in \mathcal{M}_{\la}$ if and only if $t$ is a
solution of $\psi_u(t)={\la} \int_{\Om} a(x) |u|^{q} dx$ and if $tu\in\mc M_\la$, then $\varphi_{tu}^{\prime \prime}(1)=t^{q-1}\psi_u^{\prime}(t)$.
 Moreover, we see  $\psi_u(t)\ra -\infty$ as $t \ra \infty$, $\psi_u(t)>0$ for $t$ small enough and $\psi^{\prime}_u(t)<0$ for $t$ large enough.
 Now based on the sign of $\ds \int_{\Om} a(x)|u|^{q} dx$, we will study the fibering map $\varphi_{u}$.
 \begin{Lemma}\label{lemm5}
 	Let $u(\not\equiv 0)\in X_1$ and $\la\in(0,\la_0)$.
 	\begin{enumerate}
 		\item[(i)] If $ \int_{\Om} a(x)|u|^{q} dx>0$, then there exist unique $t_1<t_{\max}<t_2$ such that $t_1u\in\mc M_\la^+$ and $t_2u\in\mc M_\la^-$. Moreover, $\mc{I}_\la(t_1u) = \ds\min_{0 \leq t \leq t_2}\mc{I}_\la (tu)$ and $\mc{I}_\la(t_2u)
 		= \ds\max_{t \geq t_{\max}} \mc{I}_\la(tu)$.
 		\item[(ii)] If $ \int_{\Om} a(x)|u|^{q} dx<0$, then there exists unique $t_2>0$ such that $t_2u\in\mc M_\la^-$.
 	\end{enumerate}
 \end{Lemma}
\begin{proof}
%\noi \textbf{Case 1}: If $\ds \int_{\Om} a(x)|u|^{q} dx>0$.\\
$(i)$ Let $u\in X_1$ such that $ \int_{\Om} a(x)|u|^{q} dx>0$. We claim that there exists unique $t_{max}>0$ such that $\psi_u^\prime(t_{\max})=0$.  To prove this, it is sufficient to show the existence of unique $t_{\max}$ such that $F_u(t_{\max})= \ba (p-q)\|u\|_{X_2}^p$, where $F_u(t):= (2r-q)t^{2r-p}\ds\int_{\Om} \int_{\Om} \frac{ |u(x)|^{r}|u(y)|^r}{|x-y|^{\mu}}~dx dy-(2-q)t^{2-p}\|u\|_{X_1}^{2}$.
%\begin{align*}
%\psi_u^\prime(t)= &(2-q)t^{1-q}\|u\|_{X_1}^{2} +\ba(p-q)t^{p-q-1}\|u\|_{X_2}^{p} - (2r-q)t^{2r-q-1}\int_{\Om} \int_{\Om} \frac{ |u(x)|^{r}|u(y)|^r}{|x-y|^{\mu}} ~dx dy\\
%=&t^{p-q-1}\bigg[ (2-q)t^{2-p}\|u\|_{X_1}^{2} +\ba(p-q)\|u\|_{X_2}^{p} - (2r-q)t^{2r-p}\int_{\Om} \int_{\Om} \frac{ |u(x)|^{r}|u(y)|^r}{|x-y|^{\mu}} ~dx dy\bigg].
%\end{align*}
%Let $M_u(t)=(2-q)t^{2-p}\|u\|_{X_1}^{2} +\ba \; (p-q)\|u\|_{X_2}^{p} - (2r-q)t^{2r-p}\ds\int_{\Om} \int_{\Om} \frac{ |u(x)|^{r}|u(y)|^r}{|x-y|^{\mu}}~dx dy$, then it is sufficient to prove that there exists unique $t_{max}>0$ such that  $M_u(t_{max})=0$. Define  $F_u(t)= (2r-q)t^{2r-p}\ds\int_{\Om} \int_{\Om} \frac{ |u(x)|^{r}|u(y)|^r}{|x-y|^{\mu}}~dx dy-(2-q)t^{2-p}\|u\|_{X_1}^{2} $, so $F_u(t_{max})- \ba(p-q)\|u\|_{X_2}^p= -M_u(t)$.
By the fact that $p<2<r$, we see  that $F_u(t)<0$ for $t$ small enough, $F_u(t)\ra\infty$ as $t\ra\infty$. Hence, there exists unique $\hat t>0$ such that $F_u(\hat t)=0$. Moreover, there exists unique $\tilde{t}>0$ such that $F_u^\prime(\tilde{t})=0$.
%Indeed $\hat t=\left(\frac{(2-q)\|u\|_{X_1}^2}{(2r-q)\ds\int_\Om \int_{\Om} \frac{ |u(x)|^{r}|u(y)|^r}{|x-y|^{\mu}}~dx dy}\right)^\frac{1}{2r-2}>0$.
 Therefore, there exists unique $t_{max}>\hat t>0$ such that $F_u(t_{max})= \ba(p-q)\|u\|_{X_p}^p$.
Using these, we conclude that $\psi_u$ is increasing in $(0,t_{max})$, decreasing in $(t_{max},\infty)$. As a consequence,
\begin{align*}
(2-q)t^{2}_{max}\|u\|_{X_1}^{2}&\leq (2-q)t^{2}_{max}\|u\|_{X_1}^{2}+ \ba (p-q)t^{p}_{max}\|u\|_{X_2}^{p} \\&= t^{2r}_{max}(2r-q)\ds\int_\Om \int_{\Om} \frac{ |u(x)|^{r}|u(y)|^r}{|x-y|^{\mu}}~dx dy     \leq (2r-q) t^{2r}_{max}C(n,\mu)S^{-r}_{\hat{r}}\|u\|_{X_1}^{2r}.
\end{align*}
Define
 \begin{equation*}
T_0 :=\frac{1}{\|u\|_{X_1}}\left(\frac{(2-q)S^r_{\hat r}}{(2r-q)C(n,\mu)}\right) ^{1/(2r-2)}\leq t_{max}
\end{equation*}
then,
\begin{align*}
\psi_u(t_{max})\geq  \psi_u(T_0)&\geq T_0^{2-q}\|u\|^{2}_{X_{1}}-T_0^{2r-q}C(n,\mu)S_{\hat r}^{-r}\|u\|^{2r}_{X_1}\\
&=\|u\|^{q}_{X_1}\left(\frac{2r-2}{2r-q}\right)\left(\frac{(2-q)S^r_{\hat r}}{(2r-q)C(n,\mu)}\right)^{\frac{2-q}{2r-2}}\geq 0.
\end{align*}
Since $\la < \la_0$, then there exist $t_1<t_{max}$ and $t_2>t_{max}$ such that $\psi_u(t_1)=\psi_u(t_2)=\la \ds \int_{\Om} a(x) |u|^{q} dx $. That is,  $t_1u, t_2u \in \mathcal{M}_{\la}$. Also $\psi_u^{\prime}(t_1)>0$ , $\psi_u^{\prime}(t_2)>0$ implies $t_1u \in \mathcal{M}^{+}_{\lambda}$ and $t_2u \in \mathcal{M}^{-}_{\la}$.  Since $\varphi^{\prime}_{u}(t) = t^{q}(\psi_u(t)- \lambda \int_{\Omega} a(x)|u|^{q}~dx)$,  $\varphi^{\prime}_{u}(t)<0$ for all $t \in [0, t_1)$ and $\varphi^{\prime}_{u}(t)>0$ for all $t \in (t_1, t_2)$. So, $\mathcal{I}_\lambda(t_1u) = \displaystyle\min_{0 \leq t \leq t_2}\mathcal{I}_\lambda (tu).$ Also, $\varphi^{\prime}_u(t) > 0$ for all $t \in [t_1, t_2),\;
\varphi^{\prime}_u(t_2) = 0$ and $\varphi^{\prime}_u(t) < 0$ for all $t \in (t_2, \infty)$ implies $\mathcal{I}_\lambda(t_2u)
= \displaystyle\max_{t \geq t_{\max}} \mathcal{I}_\lambda(tu).$\\
$(ii)$ Let $u\in X_1$ be such that $ \int_{\Om} a(x)|u|^{q} dx<0$.
From $(i)$, we have $\psi_u$ is increasing in $(0,t_{max})$, decreasing in $(t_{max},\infty)$ and $\psi_u^\prime(t_{max})=0$.  Since $\la \int_{\Om} a(x)|u|^{q} dx<0$ and $\psi_u(t_{max})>0$, there exists unique $t_1>0$ such that $\psi_u(t_1)=\la  \int_{\Om} a(x) |u|^{q} dx $ and $\psi_u^{\prime}(t_1)<0$, which implies $t_1u\in \mc{M}_{\la}^{-}$, that is $t_1u$ is a local maximum.\QED
\end{proof}

\begin{Lemma}\label{L35}
	Let $\la_0$ be defined as in \eqref{eqb10}, then the following holds.
\begin{enumerate}
	\item[(i)] There exists a constant $C_{1}>0$ such that
	$\sigma_{\la} \leq \sigma_{\la}^+ \leq - \frac{(2-q)(r-1)}{2qr} \; C_1<0$.	
	\item[(ii)] $\inf\{\|u\|: u\in\mc M_\la^- \}>0$.
\end{enumerate}
\end{Lemma}
\begin{proof}
To prove $(i)$, let $u_{0} \in \mathrm{X}_{1}$ such that $\int_\Omega{a(x)|u_{0}|^q dx}>0$. It implies there exists $ t_0=t_{0}(u_{0}) > 0$ such that  $t_{0}u_{0} \in \mathcal{M}^{+}_{\lambda}$ that is, $\varphi_{t_{0}u_{0}}^{\prime\prime}(1)>0 $. Therefore, we have
\begin{equation*}
\begin{aligned}
 \mathcal{I}_{\lambda}(t_{0}u_{0}) &\leq \left(\frac{2-q}{q}\right) \left(\frac{1}{2r}-\frac{1}{2}\right) \|t_{0} u_{0}\|_{X_1}^{2}+\ba\left(\frac{p-q}{q}\right) \left(\frac{1}{2r}-\frac{1}{p}\right) \|t_{0} u_{0}\|_{X_2}^{p}\\
&\leq \left(\frac{2-q}{q}\right) \left(\frac{1}{2r}-\frac{1}{2}\right) \|t_{0} u_{0}\|_{X_1}^{2} \leq -\frac{(2-q)(r-1)}{2qr} \; C_1,
  \end{aligned}
\end{equation*}
\noi where $C_1=\|t_0 u_0 \|_{X_1}^{2}.$ This implies $ \sigma_{\lambda}^+ \leq -\frac{(2-q)(r-1)}{2qr}\;C_{1}<0$.\\
To prove $(ii)$, let $u \in \mathcal{M}_{\la}^{-} $, then $\varphi^{\prime\prime}_u(1)<0$ which implies that
\begin{align*}
(2-q) \|u\|_{X_1}^2
& \leq (2-q) \|u\|_{X_1}^2+\ba(p-q) \|u\|_{X_2}^p \\
& < (2r-q) \int_{\Om} \int_{\Om} \frac{ |u(x)|^{r}|u(y)|^r}{|x-y|^{\mu}} ~dx dy \\
& \leq (2r-q)C(n,\mu)S^{-r}_{\hat r}\|u\|_{X_1}^{2r},
\end{align*}
that is,
  \begin{equation*}
 \|u\|_{X_1} > \left(\frac{(2-q)S^r_{\hat r}}{(2r-q)C(n,\mu)}\right)^{1/(2r-2)}.
 \end{equation*}
 Hence the result follows.
%\noi Hence, we have
%\begin{align*}
%\mathcal{I}_{\la}(u)& = \left(\frac{1}{2}-\frac{1}{2r}\right)\|u\|_{X_1}^{2}+\ba\left(\frac{1}{p}-\frac{1}{2r}\right)\|u\|_{X_2}^{p}-\la \left(\frac{1}{q}-\frac{1}{2r}\right)\int_{\Om}a(x)|u|^q~dx\\
%& \geq \left(\frac{1}{2}-\frac{1}{2r}\right)\|u\|_{X_1}^{2}-\la \left(\frac{1}{q}-\frac{1}{2r}\right)\int_{\Om}a(x)|u|^q~dx\\
%& \geq \|u\|^q_{X_1}\left(\left(\frac{1}{2}-\frac{1}{2r}\right)\|u\|_{X_1}^{2-q}-\la \left(\frac{1}{q}-\frac{1}{2r}\right)\|a\|_{\frac{d}{d-q}}S_d^{-\frac{q}{2}}\right)\\
%& > \left(\frac{(2-q)S^r_{\hat r}}{(2r-q)C(n,\mu)}\right)^{\frac{1}{r-1}}\frac{(r-1)(q-1)}{2rq}:=C_2>0 .
%\end{align*}
\QED
\end{proof}
\begin{Lemma}\label{tt}
 Let $\lambda \in (0, \lambda_{0}),$ and ${z \in \mathcal{M}_{\lambda}}$, then there exists $\epsilon > 0$ and a differentiable function
$\xi : \mathcal{B}(0,\epsilon) \subseteq X_{1} \rightarrow \mathbb{R}^{+}$ such that $\xi(0)=1,$ the function $\xi(w)(z-w)\in \mathcal{M}_{\lambda}$
and
\begin{equation*}
\langle\xi^{\prime}(0), w\rangle = \frac{2 \mc{K}_1( z, w)+p\ba\mc K_2(z,w) - q\la \int_\Om a(x)|z|^{q-2}zwdx- 2r \mc A_r(z,w) } {(2-q)\|z\|_{X_1}^{2}+(p-q)\|z\|_{X_2}^p -(2r-q)\ds \int_{\Om} \int_{\Om} \frac{ |z(x)|^{r}|z(y)|^r}{|x-y|^{\mu}} ~dx dy },
\end{equation*}
for all $w\in X_1$.
\end{Lemma}
\begin{proof}
For  ${z\in \mathcal{M}_\lambda}$, define  a function $\mathcal{H}_z:\mathbb{R}\times X_{1} \rightarrow \mathbb{R}$ given by
\begin{equation*}
\begin{aligned}
\mathcal{H}_z(t,w) &:= \langle \mathcal{I}^{\prime}_{\la}(t(z-w)),(t(z-w))\rangle\\& \;=
 t^{2}\|z-w\|_{X_1}^{2}+ \ba t^{p}\|z-w\|_{X_2}^{p}-t^{q}\la\int_\Om{a(x)|z-w|^{q}dx} \\&  \qquad -t^{2r}\ds \int_{\Om} \int_{\Om} \frac{ |(z-w)(x)|^{r}|(z-w)(y)|^r}{|x-y|^{\mu}} ~dx dy.
\end{aligned}
\end{equation*}
\noi Then $\mathcal{H}_z(1,0) = \langle \mathcal{I}^{\prime}_{\la} (z),z\rangle = 0$ and by Lemma \ref{b5}, we have
$\frac{\partial}{ \partial t}\mathcal{H}_z(1,0)\neq 0.$
\noi Therefore, by implicit function theorem result follows (for details see \cite[Lemma 3.5]{goel1}).
 \QED
\end{proof}

\begin{Proposition}\label{propb2}
Let $\lambda \in (0,\lambda_{0})$, then there exists a minimizing sequence $\{u_k\} \subset \mathcal{M}_{\lambda}$ such that
\begin{center}
    $\mathcal{I}_{\lambda}(u_{k}) = \sigma_{\lambda}+o_k(1)$ and $\mathcal{I}_{\lambda}^{\prime}(u_{k}) = o_k(1).$
\end{center}
\end{Proposition}
\begin{proof}
Using Lemma \ref{le44} and  Ekeland variational principle \cite{eke1974}, there exists a minimizing sequence $\{u_k\}\subset\mathcal{N}_\lambda $ such that
\begin{equation}\label{evp1}
\mathcal{I}_\lambda(u_k)< \sigma_\lambda+\frac{1}{k}, \mbox{ and}
\end{equation}
\begin{equation*}
\mathcal{I}_\la(u_k)< \mathcal{I}_\la(v)+\frac{1}{k}\|v-u_k\|_{X_1}\; \textrm{ for each}\; v \in \mathcal{M}_\la.
\end{equation*}
By taking $k$ large, using equation \eqref{evp1} and Lemma \ref{L35}, we deduce that
\begin{align*}
\mc I_{\la}(u_k)&= \left(\frac{1}{2}-\frac{1}{2r}\right)\|u_k\|_{X_1}^{2} +\ba \left(\frac{1}{p}-\frac{1}{2r}\right)\|u_k\|_{X_2}^{p}-\la\left(\frac{1}{q}-\frac{1}{2r}\right)\int_\Om a(x)|u_k|^{q}dx\\&< \sigma_\la+\frac{1}{k}<\sigma_\la^{+}<0,
\end{align*}
which gives us $u_k\not\equiv 0$ for $k$ large enough.
Now, using H\"older's inequality, we get
\begin{equation*}
\left(\frac{\;2qr(-\sigma_{\la}^{+})S_d^{\frac{q}{2}}}{\la(2r-q)\|a\|_{\frac{d}{d-q}}}\right)^{1/q}\leq \|u_k\|_{X_1}\leq \left(\frac{\;\la (2r-q)\|a\|_{\frac{d}{d-q}}}{q (r-1)S_d^{\frac{q}{2}}}\right)^{1/(2-q)}.
\end{equation*}
\noi Then, proof of the result $\mathcal{I}^{\prime}_\lambda(u_k)\rightarrow 0$, as $k\rightarrow \infty$ follows exactly on the same line of \cite[Proposition 4.1]{goel1}.\QED
\end{proof}
%%%%%%%%%%%%%%%%%%%%%%%%%%%%%%%%%%%%%%%%%%%
\section{Subcritical case 	(when $r<2^*_{\mu}$)}
In this section, we prove the existence of at least two non-negative solutions of problem $(P_\la)$.
\begin{Lemma}\label{lemm4}
 The functional  $\mathcal{I}_\la$ satisfies $(PS)_c$ condition for all $c \in \mathbb{R}$. That is, if $\{u_k\}\subset X_1$ satisfies
	\begin{equation}\label{eqb16}
	     \mc {I}_\la(u_k)=c+o_k(1) \;\text{and}\; \mc {I}'_\la(u_k)=o_k(1)\; \textrm{in}\; X_{1}^{\prime},
	\end{equation}
	then $\{u_k\}$ has a convergent subsequence in $X_1$.
\end{Lemma}
\begin{proof}
	Let $\{u_{k}\}\subset X_1$ satisfies \eqref{eqb16}. Then it is easy to verify that sequence $\{u_k\}$ is bounded in $X_1$. So up to subsequence $u_{k} \rightharpoonup u_{0}$ weakly in $\mathrm{X}_{1}$, $u_{k} \rightarrow u_{0}$ strongly in $\mathrm{L}^{\nu}(\Om), 1 \le \nu < 2^*_{s_1}$ and $u_{k}(x) \rightarrow u_{0}(x)$ a.e. in $\Omega$. Since $\langle\mathcal{I}_{\la}^{\prime}(u_k)-\mathcal{I}_{\la}^{\prime}(u_0), (u_{k}-u_{0})\rangle \rightarrow 0$ as $k\rightarrow \infty$, we have
	\begin{equation}\label{eqb17}
	\begin{aligned}
	  o_k(1)= &\langle\mc {I}_\la^\prime (u_k)-\mc {I}_\la^\prime(u_0), u_k-u_0\rangle \\
	    = &\mc K_1(u_k, u_k-u_0)- \mc K_1(u_0, u_k-u_0)+\ba (\mc K_2(u_k, u_k-u_0)- \mc K_2(u_0, u_k-u_0) )\\
	    & -\la\int_{\Om} a(x)\big(|u_k(x)|^{q-2}u_k(x)-|u_0(x)|^{q-2}u_0(x)\big)(u_k(x)-u_0(x))dx \\ &- \big(\mc A_r(u_k,u_k-u_0)-\mc A_r(u_0,u_k-u_0)\big).
	    \end{aligned}
	\end{equation}
 Using  H\"older inequality and the fact that $d<2^*_{s_1}$, we obtain
	%and  compactness of the imbedding $X_1\hookrightarrow L^{\alpha}(\Omega)$ for all $1<\alpha<(2)^*_{s_1}$
	\begin{equation*}
	\int_\Om{a(x)|u_{k}|^{q-2}u_k(u_k-u_0)dx} \leq \|a\|_{\frac{d}{d-q}}\|u_k\|_{d}^{q-1}\|u_k-u_0\|_{d}\rightarrow 0\;\;\mathrm{as}\;\; k \rightarrow \infty.
	\end{equation*}
	
	\noi Again, using  H\"older inequality, Hardy-Littlewood-Sobolev inequality with the fact $\hat{r}=\frac{2nr}{2n-\mu}<2^*_{s_1}$, we deduce that
	\begin{align*}
    \mc A_r(u_k,u_k-u_0)&=  \int_{\Om} \int_{\Om} \frac{ |u_k(x)|^{r}|u_k(y)|^{r-2}u_k(y)(u_k(y)-u_0(y))}{|x-y|^{\mu}} ~dx dy \\
    &\leq C(n,\mu)\|u_k\|_{\hat r}^{2r-1}\|u_k-u_0\|_{\hat r}\ra 0, \mbox{ as }k\ra\infty.
	\end{align*}
	Now, we claim that  the sequence $\{u_k\}$ has a convergent subsequence.\\
  Using the definition of $\mc K_1$, it is easy to see that $\mc K_1(u_k, u_k-u_0)- \mc K_1(u_0, u_k-u_0)=\|u_k-u_0\|_{X_1}^2$.
	Furthermore, since we know that
	\begin{align*}
	|a-b|^{\eta} \leq C_{\eta}\left((|a|^{\eta-2}a-|b|^{\eta-2}b)(a-b)\right)^{\frac{\eta}{2}}\left(|a|^\eta+|b|^{\eta}\right)^{\frac{2-\eta}{2}} \;\text{for}\; a, b \in \mathbb{R}^{n},\; 1<\eta\leq  2,
	\end{align*}
	\noi where $C_\eta$ is some positive constant depending on $\eta$.
	Set $a=u_k(x)-u_k(y)$, $b=u_0(x)-u_0(y)$ and then using H\"older inequality, we deduce that
	\begin{align*}
	\|u_k-u_\la\|_{X_2}^{p}
	&\le C (\mc K_2(u_k, u_k-u_0) - \mc K_2(u_0, u_k-u_0) )^\frac{p}{2} &\\
	&\qquad\quad \left(\ds\int_Q\frac{|u_k(x)-u_k(y)|^{p}+|u_0(x)-u_0(y)|^{p}}{|x-y|^{n+ps_2}} \right)^\frac{2-p}{2}
	\end{align*}
	and boundedness of $\{u_k\}$ in $X_2$ (follows from boundedness in $X_1$ and Lemma \ref{L1}), implies
	$$\|u_k-u_0\|_{X_2}^{p}\leq C (\mc K_2(u_k, u_k-u_0) - \mc K_2(u_0, u_k-u_0) )^\frac{p}{2}. $$
	Collecting all these informations in \eqref{eqb17}, we obtain
	\begin{align*}
	o_k(1)=\langle \mc{I}_\la^\prime(u_k)-\mc{I}_\la^\prime(u_0), u_k-u_0\rangle \ge \frac{1}{C}\big(\|u_k-u_0\|_{X_1}^2+\ba\|u_k-u_0\|_{X_2}^2\big).
	\end{align*}
	Hence, it concludes proof of the claim. \QED
\end{proof}

\noi\textbf{Proof of Theorem \ref{th1} :} Using Proposition \ref{propb2} and Lemma \ref{lemm4}, there exist minimizing sequences $\{u_k\}\in \mathcal{M}_\lambda^{+} $, $\{v_k\}\in \mathcal{M}_\lambda^{-} $ and $u_0$ and $v_0 \in X_1$ such that $u_k\rightarrow u_0$  and  $v_k\rightarrow v_0$ strongly in $X_1$ for $\la\in(0, \la_0)$. Therefore, for $\la\in(0, \la_0)$, $u_0,v_0$ are weak solutions of problem $(P_\la)$. By means of Lemma \ref{L35}, we conclude that $u_0,v_0\not\equiv 0$, hence $u_0\in \mathcal{M}_\la^{+}$ and $v_0\in \mathcal{M}_\la^{-}$. Moreover, $\mathcal{I}_\lambda(u_0)=\sigma_\lambda^{+}$ and $\mathcal{I}_\lambda(v_0)=\sigma_\lambda^{-}$. Since $\mathcal{M}_\lambda^+\cap \mathcal{M}_\lambda^-=\emptyset$, therefore $u_0$ and $v_0$ are distinct solutions.\\
Now we prove non-negativity of $u_0$. If $u_0\ge 0$, then we have a non negative solution of problem $(P_\la)$, which is also a minimizer for $\mc I_\la$ in $\mc M_\la^+$, otherwise we have $|u_0|\not\equiv 0$, hence by fibering map analysis we get unique $t_1>0$ such that $t_1u_0\in\mc M_\la^+$. We note that $\psi_{|u_0|}(1)\leq \psi_{u_0}(1)=\la\ds\int_\Om a(x)|u_0|^q= \psi_{|u_0|}(t_1)\le \psi_{u_0}(t_1)$ and $0< \psi^\prime_{u_0}(1)$, because of the fact $u_0\in\mc M_\la^+$, which implies $t_1\ge 1$. Thus,
$$\sigma_\la^+\le\varphi_{|u_0|}(t_1)\leq \varphi_{|u_0|}(1)\leq\varphi_{u_0}(1)=\sigma_{\la}^+.$$
Hence, $\mc I_\la(t_1|u_0|)=\varphi_{|u_0|}(t_1)=\sigma_\la^+$ and $t_1|u_0|\in\mc M_\la^+$ that is, $t_1|u_0|$ is a nonnegative solution of problem $(P_\la)$ in $\mc M_\la^+$. \QED

\section{Critical case 	(when $r=2^*_{\mu}$)}
In this section we assume the function $a(x)$ is continuous in $\Om$ and $a^+(x)=\max\{a(x),0\}\not\equiv 0$. Then without loss of generality we may assume there exists $\de_1>0$ such that $m_a:=\inf_{x\in B_{\de_1}}a(x)>0$.
\begin{Theorem}\label{thb2}
Let $\{u_k\}\subset \mathcal{M}_{\la}$ be a $(PS)_c$ sequence for $\mathcal{I}_{\la}$ such that $u_k \rightharpoonup u$ weakly in $X_1$, then $\mathcal{I}_{\la}^{\prime}(u)=0$. Moreover, there exists a positive constant $ C_0=C_0(q,s_1,n,S,|\Om|)$  such that $\mathcal{I}^{\prime}_{\la}(u)\geq - C_0 \la^{\frac{2}{2-q}}$,
where
\begin{equation}\label{eqb27}
C_0=\left(\frac{(2.2^*_\mu-q)(2-q)}{2.2.2^*_\mu\; q}\right) \left(\frac{2.2^*_\mu-q}{2.2^*_\mu-2}\right)^{\frac{q}{2-q}}S^{\frac{-q}{2-q}}\|a\|_{\infty}^{\frac{2}{2-q}}|\Om|^\frac{2(2^*_{s_1}-q)}{2^*_{s_1}(2-q)}.
\end{equation}
\end{Theorem}
\begin{proof}
Since $u_k \rightharpoonup  u$ in $X_1$, it implies $\{u_k\}$ is a bounded sequence in $X_1$, and up to subsequence, $u_k  \ra u \text{ in } L^{\nu}(\Om),\; 1\leq \nu < 2^*_{s_1}$ and $u_k \ra u$ a.e. in $\Om$. Now from the proof of \cite[Theorem 4.3]{goel1} it follows that
$$
\mathcal{K}_i( u_k , v )\ra\mathcal{K}_i( u , v) \text{ for } i=1,2, \text{ and } \int_{\Om}a(x)\big(|u_k(x)|^{q-2}u_k(x)-|u(x)|^{q-2}u(x)\big)v(x)\;dx \ra 0 $$
as $k\ra\infty$, for all $v \in X_1$.
From the continuous embedding of $X_1$ into $L^{2^*_{s_1}}$, we get
$u_k\rightharpoonup u \text{ weakly in }L^{2^*_{s_1}},$
as $k\rightarrow \infty$. Therefore,
 $|u_k|^{2^*_{\mu}} \rp |u|^{2^*_{\mu}}$ in $L^{2^*_{s_1}/2^*_{\mu}}(\Om)$ and we know that Riesz potential defines a continuous linear map from  $L^{2^*_{s_1}/2^*_{\mu}}(\Om)$ to $L^{\frac{2n}{\mu}}(\Om)$, thus we have
\begin{align*}
|x|^{-\mu}*|u_k|^{2^*_{\mu}}\rp |x|^{-\mu}*|u|^{2^*_{\mu}} \text{ in } L^{\frac{2n}{\mu}}(\Om).
\end{align*}
Moreover, we have $|u_k|^{2^*_{\mu}-2}u_k \rp |u|^{2^*_{\mu}-2}u$ in $L^{2^*_{s_1}/(2^*_{\mu}-1)}(\Om)$. Combining all these facts, we obtain
\begin{equation*}
\begin{aligned}
\int_{\Om} \int_{\Om} \frac{|u_k(x)|^{2^*_{\mu}}|u_k(y)|^{2^*_{\mu}-2}u_k(y) \phi(y)}{|x-y|^{\mu}}~dxdy   \ra \int_{\Om} \int_{\Om}\frac{|u(x)|^{2^*_{\mu}}|u(y)|^{2^*_{\mu}-2}u(y) \phi(y)}{|x-y|^{\mu}}~dxdy
% \int_{\Om}\left(|x|^{-\mu}*|u|^{(2)^*_{\mu}}\right) |u|^{(2)^*_{\mu}-2}u \phi ~dx
\end{aligned}
\end{equation*}

\noi Therefore, we have
\begin{equation*}
\begin{aligned}
\langle \mathcal{I}_{\la}^{\prime}(u_k)-\mathcal{I}_{\la}^{\prime}(u),\phi \rangle &= \mc K_1( u_k,\phi) -\mc K_1( u,\phi )+\ba\left(\mc K_2( u_k,\phi) -\mc K_2( u,\phi )\right)\\ &\;\;-\int_{\Om}a(x)\left(|u_k(x)|^{q-2}u_k(x)-|u(x)|^{q-2}u(x)\right)\phi(x)\;dx\\
& \;\; -\big(\mc A_{2^*_{\mu}}(u_k,\phi)-\mc A_{2^*_{\mu}}(u,\phi)\big)
   \ra 0 \qquad \text{ for } \phi \in X_1.
\end{aligned}
\end{equation*}
\noi This implies  $ \mathcal{I}_{\la}^{\prime}(u)=0$. In particular, $\langle \mathcal{I}_{\la}^{\prime}(u),u\rangle=0$, that is
\begin{equation}\label{eqb25}
\begin{aligned}
\mathcal{I}_{\la}(u) &= \left(\frac{1}{2}-\frac{1}{2 .2^*_{\mu}}\right)\|u\|_{X_1}^{2} +\ba \left(\frac{1}{p}-\frac{1}{2 .2^*_{\mu}}\right)\|u\|_{X_2}^{p} -\la\left(\frac{1}{q}-\frac{1}{2 .2^*_{\mu}}\right)\int_\Om a(x)|u|^{q}~dx\\
& \geq \left(\frac{1}{2}-\frac{1}{2 .2^*_{\mu}}\right)\|u\|_{X_1}^{2}-\la \left(\frac{1}{q}-\frac{1}{2 .2^*_{\mu}}\right)\int_\Om a(x)|u|^{q}~dx.
\end{aligned}
\end{equation}

 \noi By H\"older inequality, Sobolev embeddings and Young inequality, we obtain
\begin{equation}\label{eqb26}
\begin{aligned}
\la \int_\Omega a(x)|u|^{q}dx &\leq \la \|a\|_{\infty}S^{\frac{-q}{2}} |\Om|^{\frac{2^*_{s_1}-q}{2^*_{s_1}}} \|u\|_{X_1}^{q}\\
&= \left(\frac{2}{q}\left(\frac{1}{2}-\frac{1}{2.2^*_{\mu}}\right) \left(\frac{1}{q}-\frac{1}{2.2^*_{\mu}}\right)^{-1}\right)^\frac{q}{2}\|u\|_{X_1}^{q}\\
&\qquad \quad \la \left(\frac{2}{q}\left(\frac{1}{2}-\frac{1}{2.2^*_{\mu}}\right) \left(\frac{1}{q}-\frac{1}{2.2^*_{\mu}}\right)^{-1}\right)^\frac{-q}{2}\|a\|_{\infty} \; |\Om|^{\frac{2^*_{s_1}-q}{2^*_{s_1}}} S^{\frac{-q}{2}} \\
& \leq \left(\frac{1}{2}-\frac{1}{2.2^*_{\mu}}\right) \left(\frac{1}{q}- \frac{1}{2.2^*_{\mu}}\right)^{-1}\|u\|_{X_1}^{2}\\
&\qquad+ \left(\frac{2-q}{2}\right)\left(\frac{2.2^*_{\mu}-q}{2.2^*_{\mu}-2}\right)^{\frac{q}{2-q}}  S^{\frac{-q}{2-q}}\|a\|_{\infty}^{\frac{2}{2-q}} |\Om|^{\frac{2(2^*_{s_1}-q)}{(2-q)2^*_{s_1}}} \la ^{\frac{2}{2-q}},
\end{aligned}
\end{equation}
\noi Therefore, result follows from equations \eqref{eqb25} and \eqref{eqb26} with $$C_0= \left(\frac{1}{q}-\frac{1}{2 .2^*_{\mu}}\right)\left(\frac{2-q}{2}\right)\left(\frac{2.2^*_{\mu}-q}{2.2^*_{\mu}-2}\right)^{\frac{q}{2-q}}  S^{\frac{-q}{2-q}}\|a\|_{\infty}^{\frac{2}{2-q}} |\Om|^{\frac{2(2^*_{s_1}-q)}{(2-q)2^*_{s_1}}}.$$\QED
\end{proof}

\begin{Lemma}\label{lem3}
(Palais-Smale range). $\mathcal{I}_{\la}$ satisfies the $(PS)_c$ condition with $c\in(-\infty,c_\infty)$, where
\begin{equation*}
 c_{\infty}:=\left(\frac{n-\mu+2s_1}{2(2n-\mu)}\right) \frac{S^{\frac{2n-\mu}{n-\mu+2s_1}}}{\left(C(n,\mu)\right)^{\frac{n-2s_1}{n-\mu+2s_1}}} -C_0 \la^{\frac{2}{2-q}}
 \end{equation*}
and $C_0$ is the positive constant defined in \eqref{eqb27}.
\end{Lemma}

\begin{proof}
\noi Let $\{u_k\}$ be a $(PS)_c$ sequence of $\mathcal{I}_{\la}$ in $X_1$. Then we have
\begin{equation}\label{eqb28}
\begin{aligned}
\frac{1}{2} \|u_k\|_{X_1}^2+ \frac{\ba}{p} \|u_k\|_{X_2}^p -\frac{\la}{q}\int_{\Om}a(x)|u_k|^{q}~dx
 -\frac{1}{2 .2^*_{\mu}} \|u_k\|_{NL}^{2.2^*_{\mu}}  =c+o_k(1)
\end{aligned}
\end{equation}
and
\begin{equation}\label{eqb29}
\begin{aligned}
\|u_k\|_{X_1}^2+\ba \|u_k\|_{X_2}^p -\la\int_{\Om}a(x)|u_k|^{q} dx
-\|u_k-u\|_{NL}^{2.2^*_{\mu}} =o_k(1).
\end{aligned}
\end{equation}
\noi As an easy consequence of this, we get $\{u_k\}$ is a bounded sequence in $X_1$. Therefore, up to a subsequence, $u_k \rightharpoonup u$ in $X_1$, for some $u\in X_1$ and by Lemma \ref{thb2}, we see that $u$ is a critical point of $\mathcal{I}_{\la}$.\\
\textbf{Claim:} $u_k \ra u$ strongly in $X_1$.\\
Since, $u_k \ra u$ strongly in $L^{\nu}(\Om)$ for $1\leq \nu < 2^{*}_{s_1}$, it implies $\int_{\Om}a(x)|u_k|^{q}dx \ra \int_{\Om}a(x)|u|^{q}~dx $.
Also, by Brezis -Leib Lemma, we have
\begin{equation}\label{eqb30}
\begin{aligned}
& \|u_k\|^{p_i}_{X_i}=\|u_k-u\|^{p_i}_{X_i}+\|u\|^{p_i}_{X_i}+ o_k(1),\quad 1\leq i \leq 2, \; p_1=2,\; p_2=p \quad  \text {and } \\&
\|u_k\|_{NL}^{2.2^*_{\mu}}=\|u\|_{NL}^{2.2^*_{\mu}}+ \|u_k-u\|_{NL}^{2.2^*_{\mu}}+ o_k(1).
\end{aligned}
\end{equation}
Therefore, by using equations \eqref{eqb28}, \eqref{eqb29} and \eqref{eqb30}, we get
\begin{equation}\label{eqb32}
\begin{aligned}
& \frac{1}{2}\|u_k-u\|^{2}_{X_1}+ \frac{\ba}{p}\|u_k-u\|^{p}_{X_2} -\frac{1}{2 .2^*_{\mu}} \|u_k-u\|_{NL}^{2.2^*_{\mu}} =c-\mathcal{I}_{\la}(u)+ o_k(1)\\
&\|u_k-u\|^{2}_{X_1}+\ba \|u_k-u\|^{p}_{X_2} -\|u_k-u\|_{NL}^{2.2^*_{\mu}}= o_k(1).
\end{aligned}
\end{equation}

\noi Hence, let $\|u_k-u\|^{2}_{X_1}+\ba \|u_k-u\|^{p}_{X_2} \ra l$ and  $\|u_k-u\|_{NL}^{2.2^*_{\mu}}  \ra l,  \text{ as } k \ra \infty$.
If $l=0$, then claim is proved. So, we assume $l>0$, then
\begin{align*}
l^{1/2^*_{\mu}} = \big(\lim_{k \ra \infty}\|u_k-u\|_{NL}^{2.2^*_{\mu}}  \big)^{1/ 2^*_{\mu}}
 \leq \ds \left(C(n,\mu)\right)^{1/2^*_{\mu}}
\lim_{k \ra \infty}\left( S^{-1}\|u_k-u\|_{X_1}^{2}\right)
%& \leq\ds \left(C(n,\mu)\right)^\frac{1}{2^*_{\mu}} S^{-1} \lim_{k \ra \infty}\left(\|u_k-u\|_{X_1}^{2}+\ba \|u_k-u\|_{X_2}^p\right) \\
 \leq \left(C(n,\mu)\right)^\frac{1}{2^*_{\mu}} S^{-1} l.
\end{align*}
This implies $l \geq C(n,\mu)^{\frac{-1}{2^*_{\mu}-1}}S^{\frac{2^*_{\mu}}{2^*_{\mu}-1}}$, that is, $l\ge \frac{S^{\frac{2n-\mu}{n-\mu+2s_1}}}{\left(C(n,\mu)\right)^{\frac{n-2s_1}{n-\mu+2s_1}}}  $.
\noi Now, from \eqref{eqb32}, we have
\begin{align*}
c-\mathcal{I}_{\la}(u) \geq \frac{1}{2}\big(\|u_k-u\|^{2}_{X_1}+\ba \|u_k-u\|^{p}_{X_2}\big) -\frac{1}{2 .2^*_{\mu}} \|u_k-u\|_{NL}^{2.2^*_{\mu}}
 = \left(\frac{n-\mu+2s_1}{2(2n-\mu)}\right)l.
\end{align*}
Therefore,  with the help of Theorem \ref{thb2}, we get
\begin{align*}
c\geq \left(\frac{n-\mu+2s_1}{2(2n-\mu)}\right)l +\mathcal{I}_{\la}(u)
\geq\left(\frac{n-\mu+2s_1}{2(2n-\mu)}\right) \frac{S^{\frac{2n-\mu}{n-\mu+2s_1}}}{\left(C(n,\mu)\right)^{\frac{n-2s_1}{n-\mu+2s_1}}} -C_0 \la^{\frac{2}{2-q}}=c_\infty,
\end{align*}
 which is a contradiction. \QED
\end{proof}

\noi\textbf{Proof of Theorem \ref{th2}} (i): Let $\ga_0>0$ be such that for all $\la\in(0,\ga_0)$, $c_\infty>0$ holds. Set $\La=\min\{\ga_0,\la_0\}>0$.
By Proposition \ref{propb2},  for all $\la\in(0,\La)$, there exists a minimizing sequence $u_k \in \mathcal{M}_{\la}$  such that
 $\mathcal{I}_{\la}(u_{k}) = \sigma_{\la}+o_k(1)$ and $\mathcal{I}_{\la}^{\prime}(u_{k}) = o_k(1)$. Now using the fact $\sigma_{\la}\leq \sigma_{\la}^{+}<0$ and applying Lemma \ref{lem3}, there exists a $u_\la \in X_1$ such that $u_k\ra u_\la$ in $X_1$. By Theorem \ref{thb2} and due to the fact $\sigma_\la<0$, we get $\mathcal{I}_{\la}^{\prime}(u_\la) =0$ and $\int_{\Om} a(x)|u_\la|^q>0$ and hence $u_\la\not\equiv 0$, thus $u_\la\in\mc M_\la$. Next, we will prove that $u_\la\in\mc M_\la^+$. Suppose on the contrary $u_\la\in\mc M_\la^-$, then by the fibering map analysis there exist $t_1<t_2=1$ such that $t_1u_\la\in\mc M_\la^+$ and $t_2u_\la\in\mc M_\la^-$.  Since $\phi_{u_\la}$ is increasing in $[t_1,t_2)$, it implies
 $$\sigma_\la\leq \mc I_\la(t_1u_\la)<\mc I_\la(tu_\la)\le\mc I_\la(u_\la)=\sigma_\la$$
 for $t\in(t_1,1)$, which is a contradiction. Hence $u_\la\in\mc M_\la^+$ and $\sigma_\la=\mc I_\la(u_\la)=\sigma_\la^+$. Moreover, by using same assertions and arguments as in proof of Theorem \ref{th1}, we obtain that $u_\la$ is nonnegative.\QED

\noi Consider the family of functions $U_\e(x)=\e^{-\frac{(n-2s_1)}{2}}u^*(\frac{x}{\e})$,
where $u^*(x)=\bar{u}\Big(\frac{x}{S^{\frac{1}{2s_1}}}\Big), \bar{u}(x)=\frac{\tilde{u}(x)}{\|\tilde{u}\|_{2^*_{s_1}}}$
	and $\tilde{u}(x)=d(b^2+|x|^2)^{-\frac{n-2s_1}{2}}$, $d\in\mb R\setminus\{0\}$ and $b>0$. (\cite{serva}) $U_\e$ satisfies
	\begin{align}\label{eqbm}
		(-\De)^{s_1}U_\e=|U_\e|^{2^*_{s_1}-2}U_\e, \mbox{ in }\mb R^n, \mbox{ and } \|U_\e\|_{X_1}^2=\|U_\e\|_{2^*_{s_1}}^{2^*_{s_1}}=S^\frac{n}{2s_1}.
	\end{align}
Let $\eta\in C_c^\infty(\Om)$ be such that $\eta=1$ in $B_\de(0)$, $\eta=0$ in $B^c_{2\de}(0)$, and $0\le\eta\le 1$ in $\Om$. Set $u_\e=\eta U_\e$, then we have the following estimates

\begin{Lemma}\label{lemm1}(see \cite{giac,serva})
	The following hold true
	\begin{enumerate}
		\item[(i)] $\|u_\e\|_{X_1}^2\le S^\frac{n}{2s_1}+O(\e^{n-2s_1})$%$=\big(C(n,\mu)^\frac{1}{2^*_{s_1}}S_H\big)^\frac{n}{2s_1}+O(\e^{n-2s_1})$
		\item[(ii)] $\|u_\e\|_{2^*_{s_1}}^{2^*_{s_1}}=S^\frac{n}{2s_1}+O(\e^n)$
		\item[(iii)] $\|u_\e\|_{NL}^2\leq C(n,\mu)^{\frac{n}{2.2^*_{\mu}s_1}}S_H^\frac{n-2s_1}{2s_1}+O(\e^n)$
		\item[(iv)] $\|u_\e\|_{NL}^2\geq \big(C(n,\mu)^{\frac{n}{2s_1}}S_H^\frac{2n-\mu}{2s_1}-O(\e^n)\big)^\frac{1}{2^*_{\mu}}$.
	\end{enumerate}
\end{Lemma}
{ Now, we will estimate the fractional $p$-Laplacian norm of the family of functions $\{u_\e \}$.
\begin{Lemma}\label{L2}
  Let $1<p< 2$, then there exists $C>0$ such that
    \begin{align}\label{eqb56}
    \|u_\e\|_{X_2}^p \leq C \begin{cases}
      \e^{\frac{n-2s_1}{2}p} \; \; \mbox{if }1<p<n/(n-s_1) \\
      \e^{\frac{2-p}{2}n} \; \; \;\; \mbox{if }n/(n-s_1)\leq p<2.
    \end{cases}
    \end{align}
\end{Lemma}
\begin{proof}
  Following the ideas of \cite[Proposition 21]{serva}, we define the following sets
  \begin{align*}
  	&D:=\{ (x,y)\in\mb R^{2n} \ : x\in B_\de, y\in B_{\de}^c, |x-y|> \de/2    \} \; \; \mbox{and} \\
  	&E:= \{ (x,y)\in\mb R^{2n} \ : x\in B_\de, y\in B_{\de}^c, |x-y|\le \de/2 \}.
  \end{align*}
  Then, by the definition of $u_\e$, it is clear that
  \begin{align}\label{pfrac}
  	\|u_\e\|_{X_2}^p = \int_{\mb R^{2n}} \frac{|u_\e(x)-u_\e(y)|^p}{|x-y|^{n+ps_2}}dxdy &= \int_{B_\de}\int_{B_\de} \frac{|U_\e(x)-U_\e(y)|^p}{|x-y|^{n+ps_2}}dxdy  \nonumber\\
  	&\quad+ \left( 2\int_D + 2\int_E + \int_{B_{\de}^c} \int_{B_{\de}^c}  \right) \frac{|u_\e(x)-u_\e(y)|^p}{|x-y|^{n+ps_2}}dxdy.
  \end{align}
  By \cite[Claim 10 of Proposition 21]{serva}, for $x,y\in B_{\de}^c$, we have
  \begin{align*}
  	|u_\e(x)-u_\e(y)|\le C \e^{(n-2s_1)/2} \min\{ 1, |x-y|\}.
  \end{align*}
  Therefore,
  \begin{align*}
    I_1= \int_{B_{\de}^c} \int_{B_{\de}^c} \frac{|u_\e(x)-u_\e(y)|^p}{|x-y|^{n+ps_2}}dxdy \le C \e^{\frac{n-2s_1}{2}p} \int_{B_{2\de}} \int_{\mb R^n} \frac{ \min\{1,|x-y|^p\}}{|x-y|^{n+ps_2}}dxdy = C \e^{\frac{n-2s_1}{2}p}.
  \end{align*}
  Next, by \cite[(4.21)]{serva},  for $x\in B_{\de}$, $y\in B_{\de}^c$ with $|x-y|\le \de/2$, we have
  \[  	|u_\e(x)-u_\e(y)|\le C \e^{(n-2s_1)/2}|x-y|.   \]
  Then, proceeding similarly, we get
  \begin{align*}
  	I_2= \int_E \frac{|u_\e(x)-u_\e(y)|^p}{|x-y|^{n+ps_2}}dxdy \le C \e^{\frac{n-2s_1}{2}p}.
  \end{align*}
  To evaluate $I_3:= \int_D \frac{|u_\e(x)-u_\e(y)|^p}{|x-y|^{n+ps_2}}dxdy$, we first note that due to the fact $1<p<2$, there exists $A_p>0$ such that
  \begin{align*}
  	|u_\e(x)-u_\e(y)|^p \le |U_\e(x)-U_\e(y)|^p + |U_\e(y)-u_\e(y)|^p + A_p |U_\e(x)-U_\e(y)|^{p-1} |U_\e(y)-u_\e(y)|.
  \end{align*}
  Now, using $|u_\e(y)|\le |U_\e(y)|\le C \e^{(n-2s_1)/2}$ for all $y\in B_{\de}^c$ (\cite[(4.17)]{serva}), we obtain
  \begin{align}\label{eqb54}
  	\int_D \frac{|U_\e(y)-u_\e(y)|^p}{|x-y|^{n+ps_2}}dxdy\le 2^p \int_D \frac{|U_\e(y)|^p}{|x-y|^{n+ps_2}}dxdy &\le C \e^{\frac{n-2s_1}{2}p} 	\int_D \frac{1}{|x-y|^{n+ps_2}}dxdy \nonumber \\
  	&= C \e^{\frac{n-2s_1}{2}p}.
  \end{align}
  Furthermore,
  \begin{align}\label{eqb55}
  	\int_D \frac{|U_\e(x)-U_\e(y)|^{p-1} |U_\e(y)-u_\e(y)|}{|x-y|^{n+ps_2}}dxdy \le 2\int_D \left( \frac{|U_\e(x)|^{p-1} |U_\e(y)|}{|x-y|^{n+ps_2}}+  \frac{|U_\e(y)|^p}{|x-y|^{n+ps_2}} \right)dxdy.
  \end{align}
  By \cite[(4.17)]{serva} and definition of $U_\e$, we have
  \[ |U_\e(x)|^{p-1} |U_\e(y)|\le C \e^\frac{n-2s_1}{2} \left(b^2 +\big| \frac{x}{\e S^{1/(2s_1)}}  \big|^2   \right)^{-\frac{n-2s_1}{2}(p-1)}.  \]
  Therefore, for $\de_e=\de/\e$, $\xi= \frac{x}{\e S^{1/(2s_1)}}$ and $\zeta= |x-y|$, we deduce that
  \begin{align*}
  	\int_D \frac{|U_\e(x)|^{p-1} |U_\e(y)|}{|x-y|^{n+ps_2}} dxdy &\le C  \e^\frac{n-2s_1}{2}  \e^n \int_{\xi\in B_{\de_\e}} \int_{|\zeta|>\de/2} \big(b^2+|\xi|^2\big) ^{-\frac{n-2s_1}{2}(p-1)} |\zeta|^{-(n+ps_2)}d\xi d\zeta \nonumber \\
  	& \le C \e^{n+\frac{n-2s_1}{2}} \left[  1+\int_{\xi\in B_{\de_\e}\setminus B_1} \big(b^2+|\xi|^2\big) ^{-\frac{n-2s_1}{2}(p-1)} \right] \nonumber \\
  	 &\le C \e^{n+\frac{n-2s_1}{2}}  \big( \e^{(n-2s_1)(p-1)-n} +1\big) \le C \e^{(n-2s_1)(p-1/2)}.
  \end{align*}
  Using this together with \eqref{eqb54} and the fact that $(n-2s_1)/2 \le (n-2s_1)(p-1/2)$, \eqref{eqb55} implies
  \begin{align*}
  		\int_D \frac{|U_\e(x)-U_\e(y)|^{p-1} |U_\e(y)-u_\e(y)|}{|x-y|^{n+ps_2}}dxdy \le C  \e^{\frac{n-2s_1}{2}p}.
  \end{align*}
  Thus, from \eqref{pfrac}, we obtain
  \begin{align}\label{eqb59}
  	\|u_\e\|_{X_2}^p = \int_{\mb R^{2n}} \frac{|u_\e(x)-u_\e(y)|^p}{|x-y|^{n+ps_2}}dxdy &\le \int_{\mb R^{2n}} \frac{|U_\e(x)-U_\e(y)|^p}{|x-y|^{n+ps_2}}dxdy + C  \e^{\frac{n-2s_1}{2}p}.
  \end{align}
  Now, it remains to evaluate only $\|U_\e\|_{X_2}^p$. We recall $U_\e(x)= \e^{-\frac{n-2s_1}{2}} u^*\big(\frac{x}{\e}\big)$ and $U_1(x)=u^*(x)$. Using suitable change of variables, we deduce that
  \begin{align*}
  \|U_\e\|_{X_2}^p = \e^{-\frac{n-2s_1}{2}p} \int_{\mb R^{2n}} \frac{|u^*(\frac{x}{\e})-u^*(\frac{y}{\e})|^p}{|x-y|^{n+ps_2}}dxdy= \e^{-\frac{n-2s_1}{2}p} \int_{\mb R^{2n}} \frac{|u^*(z)-u^*(w)|^p}{|z-w|^{n+ps_2}}dzdw,
  \end{align*}
  then, using Lemma \ref{L1} and \eqref{eqbm}, we get
  \begin{align}\label{eqb60}
  	\|U_\e\|_{X_2}^p \leq C \e^{-\frac{n-2s_1}{2}p} \|U_1\|_{X_1}^p \le C \e^{n(1-\frac{p}{2})}.
  \end{align}
  Hence, \eqref{eqb59} and \eqref{eqb60} give the required result of the Lemma. \QED
\end{proof}

%\noi Following the ideas of \cite[Lemma 4.7]{goel1}, we give an upper bound for $\sigma_\la^-$.
\begin{Lemma}\label{lemm2}
	There exists $\La_{0}>0$ such that for every $\la\in(0,\La_{0})$ and $\ba>0$, there holds
	 $\sigma_{\la}^{-} <c_\infty, $
	 where $c_\infty$ is as defined in Lemma \ref{lem3}, in each of the following cases
	 \begin{enumerate}
	 	\item[(I)] for all $q>0$, if $1<p<n/(n-s_1)$,
	 	\item[(II)] for all $q\in\big(1, \frac{n(2-p)}{n-2s_1}\big)  \cup\big (\frac{4n}{4n-np-4s_1},p\big)$, if $n/(n-s_1)\le p<2$.
	 \end{enumerate}
\end{Lemma}
\begin{proof} To prove the lemma, we will show that there exists $\e>0$ such that $\ds\sup_{t\geq 0} \mc{I}_\la(tu_\e) <c_\infty$. Let $\ga_0>0$ be such that for all $\la\in(0,\ga_0)$, $c_\infty>0$ holds. Then, with the help of Lemma \ref{L1}, we have
	\begin{align*}
	\mc{I}_\la(tu_\e)&\leq \frac{t^{2}}{2}\|u_\e\|_{X_1}^{2}+ \ba \frac{t^{p}}{p}\|u_\e\|_{X_2}^{p} \\ &\leq
	\frac{t^{2}}{2}\|u_\e\|_{X_1}^{2}+C\ba \frac{t^{p}}{p}\|u_\e\|_{X_1}^{2}  \leq
	C(t^{2}+t^{p}).
	\end{align*}
	Therefore, there exists $t_0>0$ such that $$\ds\sup_{0\leq t\leq t_0}\mc{I}_\la(tu_\e) < c_\infty.$$
	Let $\varTheta(t)= \frac{t^{2}}{2}\|u_\e\|_{X_1}^{2} +\ba \frac{t^{p}}{p}\|u_\e\|_{X_2}^{p}-\frac{t^{2.2^*_{\mu}}}{2.2^*_{\mu}} \|u_\e\|_{NL}^{2.2^*_{\mu}},$ then
	we note that $\varTheta(0)=0$, $\varTheta(t)>0$ for $t$ small enough, $\varTheta(t)<0$ for $t$ large enough,
	and there exists $t_\e >0$ such that $\ds\sup_{t\ge 0}\varTheta(t)=\varTheta(t_\e),$ that is
	$$0=\varTheta^\prime(t_\e)=t_\e\|u_\e\|_{X_1}^{2} +\ba t_\e^{p-1}\|u_\e\|_{X_2}^{p}-t_\e^{2.2^*_{\mu}-1} \|u_\e\|_{NL}^{2.2^*_{\mu}} $$
	which gives us
	\begin{align*}
	t_\e ^{2.2^*_{\mu}-p}&=\frac{1}{\|u_\e\|_{NL}^{2.2^*_{\mu}}} \big(t_\e^{2-p}\|u_\e\|_{X_1}^{2}+ \ba\|u_\e\|_{X_2}^{p}\big)
	< C(1+t_\e^{2-p}).
	\end{align*}
	Since $2.2^*_{\mu}> 2$, there exists $t_1> 0$ such that $t_\e \le t_1$ for all  $\e >0$.
	Also, for $\de>0$ such that $2\de<\de_1$, we have
	\begin{align*}
	\ds\int_\Om a(x)|u_\e|^q dx &=\ds\int_{B_{2\de}(0)}a(x)|u_\e|^q dx
	\ge m_a\ds\int_{B_{2\de}(0)}|u_\e|^q dx
	\ge m_a\ds\int_{B_{\de}(0)}|U_{\e}|^q dx.
	\end{align*}
	Therefore, we  obtain
	\begin{align}\label{eqb33}
	\ds\sup_{t\ge t_0}\mc{I_\la}(tu_\e)&\le \ds\sup_{t\ge 0}\varTheta(t) -\frac{t_0^qm_a}{q}\la\ds\int_{B_\de(0)}|U_\e|^q    \nonumber\\
	&=\frac{t_\e^{2}}{2}\|u_\e\|_{X_1}^{2}+ \ba \frac{t_\e^{p}}{p}\|u_\e\|_{X_2}^{p} -\frac{t_\e^{2.2^*_{\mu}}}{2.2^*_{\mu}}\|u_\e\|_{NL}^{2.2^*_{\mu}}- \frac{t_0^qm_a}{q}\la\ds\int_{B_\de(0)}|U_\e|^q    \nonumber\\
	&\leq \ds\sup_{t\ge 0}\left(\frac{t^{2}}{2}\|u_\e\|_{X_1}^{2}- \frac{t^{2.2^*_{\mu}}}{2.2^*_{\mu}}\|u_\e\|_{NL}^{2.2^*_{\mu}}\right)
   +\ba \frac{t_1^{p}}{p}\|u_\e\|_{X_2}^{p}  -\frac{t_0^qm_a}{q}\la\ds\int_{B_\de(0)}|U_\e|^q.
	\end{align}
	Let $\Upsilon(t)=\frac{t^{2}}{2}\|u_\e\|_{X_1}^{2}- \frac{t^{2.2^*_{\mu}}}{2.2^*_{\mu}}\|u_\e\|_{NL}^{2.2^*_{\mu}}$. It is easy to verify that $\Upsilon$ attains its maximum at $\hat{t}=\left(\frac{\|u_\e\|_{X_1}^{2}}{\|u_\e\|_{NL}^{2.2^*_{\mu}}}\right)^{1/(2^*_{\mu}-2)}$.
	Therefore, \begin{align*}
	\ds\sup_{t\ge 0}\Upsilon(t)= \Upsilon(\hat{t})= \frac{n-\mu+2s_1}{2(2n-\mu)}	\left(\frac{\|u_\e\|_{X_1}^{2}}{\|u_\e\|_{NL}^{2.2^*_{\mu}}}\right)^{2^*_{\mu}/(2^*_{\mu}-1)},
	\end{align*}
	using the estimates of Lemma \ref{lemm1} and \eqref{eqb9}, we have
	\begin{align}\label{eqb36}
	\ds\sup_{t\ge 0}\Upsilon(t)&\leq \frac{n-\mu+2s_1}{2(2n-\mu)} \left(\frac{(C(n,\mu)^\frac{1}{2^*_{\mu}}S_H)^\frac{n}{2s_1}+O(\e^{n-2s_1})}{\big(C(n,\mu)^\frac{n}{2s_1}S_H^\frac{2n-\mu}{2s_1}-O(\e^n)\big)^\frac{1}{2^*_{\mu}}} \right)^{2^*_{\mu}/(2^*_{\mu}-1)}  \nonumber \\
	&\leq \frac{n-\mu+2s_1}{2(2n-\mu)} S_H^\frac{2n-\mu}{n-\mu+2s_1}+O(\e^{n-2s_1}).
	\end{align}
	For $\de>0$, sufficiently small such that $B_{2\de}(0)\Subset\Om$, $2\de <\de_1$ and $0<\e<\de/2$, we have the following estimate (\cite{giac})
	\begin{align}\label{eqb37}
		\int_{B_{\de}(0)}|U_\e|^q\ge C_3\begin{cases}
		      \e^{n-\frac{(n-2s_1)}{2}q}, \quad \mbox{if }n<(n-2s_1)q \\
		      \e^\frac{n}{2}|\ln\e|, \quad \mbox{if }n=(n-2s_1)q \\
		      \e^{\frac{(n-2s_1)}{2}q}, \quad \mbox{if }n>(n-2s_1)q.
		\end{cases}
	\end{align}
	Thus, using \eqref{eqb56}, \eqref{eqb36} and \eqref{eqb37} in \eqref{eqb33}, we get
	\begin{align}\label{eqb35}
	\ds\sup_{t\ge t_0}\mc{I}_\la(tu_\e)\leq \frac{n-\mu+2s_1}{2(2n-\mu)} S_H^\frac{2n-\mu}{n-\mu+2s_1} &+C_1\e^{n-2s_1}+C_4 \begin{cases}
	\e^{\frac{n-2s_1}{2}p} \; \; \mbox{if }1<p<n/(n-s_1) \\
	\e^{\frac{2-p}{2}n} \; \; \;\; \mbox{if }n/(n-s_1)\leq p<2.
	\end{cases} \nonumber\\
	&\qquad- \la C_2\begin{cases}
	\e^{n-\frac{(n-2s_1)}{2}q}, \quad \mbox{if }n<(n-2s_1)q \\
	\e^\frac{n}{2}|\ln\e|, \quad \mbox{if }n=(n-2s_1)q \\
	\e^{\frac{(n-2s_1)}{2}q}, \quad \mbox{if }n>(n-2s_1)q.
	\end{cases}
	\end{align}
	Now, we consider the following cases.\\
\noi\textbf{Case}(I): If $1<p<n/(n-s_1)$.\\
   In this case ${(n-2s_1)p/2}<n(1-p/2)$ and $1<q<p<n/(n-s_1)<n/(n-2s_1)$, therefore \eqref{eqb35} implies
   \begin{align*}
   	\sup_{t\ge t_0}\mc{I}_\la(tu_\e)\leq \frac{1}{2}\frac{n-\mu+2s_1}{2n-\mu} S_H^\frac{2n-\mu}{n-\mu+2s_1} + C_4 \e^{\frac{n-2s_1}{2}p} - \la C_2 \e^{\frac{(n-2s_1)}{2}q}.
   \end{align*}
   And,  therefore there exists $\ga_1>0$ such that for all $\la\in(0,\ga_1)$
      \begin{align*}
      	C_4\e^{\frac{n-2s_1}{2}p}-\la C_2\e^{\frac{n-2s_1}{2}q}< -C_0 \la^{\frac{2}{2-q}}.
      \end{align*}
 \textbf{Case}(II): If $n/(n-s_1)\leq p<2$.\\
 In this case, ${(n-2s_1)p/2}\ge n(1-p/2)$, therefore \eqref{eqb35} implies
    \begin{align}\label{eqb38}
    \sup_{t\ge t_0}\mc{I}_\la(tu_\e)\leq \frac{n-\mu+2s_1}{2(2n-\mu)} S_H^\frac{2n-\mu}{n-\mu+2s_1} +C_4
    \e^{\frac{2-p}{2}n}- \la C_2\begin{cases}
    \e^{n-\frac{(n-2s_1)}{2}q}, \quad \mbox{if }n<(n-2s_1)q \\
    \e^\frac{n}{2}|\ln\e|, \quad \mbox{if }n=(n-2s_1)q \\
    \e^{\frac{(n-2s_1)}{2}q}, \quad \mbox{if }n>(n-2s_1)q.
    \end{cases}
    \end{align}
    \textbf{Subcase}(a): If $n>(n-2s_1)q$.\\
    In this case, we see that $\frac{(n-2s_1)}{2}q< n(1-\frac{p}{2})$, if $q<\frac{n(2-p)}{n-2s_1}$. Then, proceeding similar to Case(I), there exists $\ga_2>0$ such that for all $\la\in (0,\ga_2)$
     \begin{align*}
    C_4\e^{\frac{n-2s_1}{2}p}-\la C_2\e^{\frac{n-2s_1}{2}q}< -C_0 \la^{\frac{2}{2-q}}.
    \end{align*}
    \textbf{Subcase}(b): If $n<(n-2s_1)q$.\\
	Choosing $\e=(\la^\frac{2}{2-q})^\frac{2}{n(2-p)}\le\de$, \eqref{eqb38} yields
	\begin{align*}
	\sup_{t\ge t_0}\mc{I}_\la(tu_\e) \leq \frac{n-\mu+2s_1}{2(2n-\mu)} S_H^\frac{2n-\mu}{n-\mu+2s_1}+C_4\la^\frac{2}{2-q}
- \la C_2
	\la^{\frac{4}{n(2-q)(2-p)}\big(n-\frac{(n-2s_1)}{2}q\big)}.
	\end{align*}
	Clearly,
	\begin{align*}
	1&+\frac{4}{n(2-q)(2-p)} \left(n-\frac{n-2s_1}{2}q\right)
	< \frac{2}{2-q}	
	\end{align*}
	if  $q>4n/(4n-np-4s_1)$.
	Then, in this situation, there exists $\ga_3>0$ such that for all $\la\in(0,\ga_3)$,
	$$C_4\la^\frac{2}{2-q}-C_2\;\la\; \la^{\frac{4}{n(2-q)(2-p)}  \left(n-\frac{n-2s_1}{2}q\right)}< -C_0\la^\frac{2}{2-q}.$$
	Let $\La_{0}=\min\{\ga_1, \ga_2, \ga_3,\La,(\frac{\de}{2})^{n(2-p)/2}\} >0$, then for all $\la\in(0,\La_{0})$ and sufficiently small $\e>0$, we obtain
	$$\ds\sup_{t\ge 0}\mc{I}_\la(tu_\e)< c_\infty.$$
	Now choosing $\de>0$ sufficiently small, we see that $u_\e\in X_1$ and using Lemma \ref{lemm5}, there exists $\tilde{t}>0$ such that $\tilde{t}u_\e\in\mc M_\la^-$.
	Hence, \begin{align*}
	\sigma_\la^- \leq\mc I_\la(\tilde{t}u_\e)\leq \ds\sup_{t\ge 0}\mc I_\la(tu_\e)< c_\infty.
	\end{align*}
	This completes proof of the Lemma.  \QED
	\end{proof}	

\noi\textbf{Proof of Theorem \ref{th2}} (ii): With the help of Lemma \ref{tt} and Proposition \ref{propb2}, we get a minimizing sequence $\{v_k\}\subset\mc M_\la^-$, which is also $(PS)_{\sigma_\la^-}$ sequence. By Lemmas \ref{lemm2} and \ref{lem3}, there exists $v_\la\in X_1$ such that $v_k\ra v_\la$ in $X_1$. Using Theorem \ref{thb2} and strong convergence, $v_\la\in\mc M_\la^-$ and $\mc I_\la(v_\la)=\sigma_\la^-$. Thus, $v_\la$ is a weak solution of problem $(P_\la)$ in $\mc M_\la^-$ and since $u_\la\in\mc M_\la^+$, $u_\la\neq v_\la$.\QED

\begin{Lemma}\label{lemm3}
	There exists $\La_{00},\ba_{00}>0$ such that for every $\la\in(0,\La_{00})$ and $\ba\in(0,\ba_{00})$, there holds
	$$\sigma_{\la}^{-} <c_\infty. $$
	%where $c_\infty$ is as defined in Lemma \ref{lem3}.
\end{Lemma}
\begin{proof}
	Here we only give outline of the proof because arguments are similar to the previous lemma and \cite[Lemma 4.7]{goel1}. Let $\ba= \e^\al$, where $\al>(n-2s_1)$ and taking into account \eqref{eqb36} and \eqref{eqb37}, from \eqref{eqb33}, it follows that
	\begin{align*}
	\sup_{t\ge t_0}\mc{I}_\la(tu_\e)\leq \frac{n-\mu+2s_1}{2(2n-\mu)} S_H^\frac{2n-\mu}{n-\mu+2s_1} +C_1\e^{n-2s_1}- \la C_2\begin{cases}
	\e^{n-\frac{(n-2s_1)}{2}q}, \quad \mbox{if }n<(n-2s_1)q \\
	\e^\frac{n}{2}|\ln\e|, \quad \mbox{if }n=(n-2s_1)q \\
	\e^{\frac{(n-2s_1)}{2}q}, \quad \mbox{if }n>(n-2s_1)q.
	\end{cases}
	\end{align*}
	The case, when $n>(n-2s_1)q$, follows exactly on the same lines of Case(I) of Lemma \ref{lemm2}. For the other case, we set $\e=(\la^\frac{2}{2-q})^\frac{1}{n-2s_1}\le\de$ and proceed similarly to \cite[Lemma 4.7]{goel1}, to get the required result of the lemma for some $\La_{00},\ba_{00}>0$.\QED
\end{proof}

\noi\textbf{Proof of Theorem \ref{th3}}: Proof follows exactly on the same lines of Proof of Theorem \ref{th2} (ii) by using Lemma \ref{lemm3} instead of Lemma \ref{lemm2}.\QED

}

\end{document}